\documentclass[final,oneeqnum]{siamltex}

\usepackage{amssymb,amsmath}
\usepackage[latin1]{inputenc}

\newcommand{\ouv}{\Omega}

\newcommand{\ouvTM}{\mathcal{O}}
\newcommand{\ouvxiBig}{\mathcal{U}}
\newcommand{\ouvgraph}{U}
\newcommand{\ouvxiProofLem}{V} 
\newcommand{\ouvxiProofLemB}{W} 

\newcommand{\prolto}[2]{{#1}_{#2}}

\newcommand{\CCC}[1]{\ensuremath{\mathbf{C}^{#1}}}
\newcommand\RR{{I\!\!R}}

\renewcommand\emptyset{\varnothing}  

\newtheorem{remark}[theorem]{Remark}

\title{A necessary condition for dynamic equivalence}
\author{Jean-Baptiste Pomet\thanks{ 
 INRIA, B.P. 93, 06902 Sophia Antipolis cedex,
    France. \hfill $\ $
Email:~\texttt{Jean-Baptiste.Pomet@sophia.inria.fr}. \hfill  May 6, 2008, revised \today.  }}

\begin{document}

\pagestyle{myheadings} \thispagestyle{plain} \markboth{J.-B. POMET}{A NECESSARY CONDITION FOR DYNAMIC EQUIVALENCE}
\maketitle
\begin{abstract}
  If two control systems on manifolds of the same dimension are dynamic equivalent, we prove that either they are static
  equivalent --\textit{i.e.}\ equivalent via a classical diffeomorphism-- or they are both ruled;  for systems of different
  dimensions, the one of higher dimension must be ruled.
  A ruled system is one whose equations define at each point in the state manifold, a ruled submanifold of the tangent
  space.
  Dynamic equivalence is also known as equivalence by endogenous dynamic feedback, or by a Lie-Bäcklund transformation when
  control systems are viewed as underdetermined systems of ordinary differential equations; it is very close to absolute equivalence for Pfaffian systems.

  It was already known that a differentially flat system must be ruled; this was a particular case of the present result, in which
  one of the systems was assumed to be ``trivial'' (or linear controllable).
\end{abstract}


\section{Introduction}

We consider time-invariant control systems, or underdetermined systems of ordinary differential equations (ODEs) where the
independent variable is time. 
\emph{Static equivalence} refers to equivalence via a diffeomorphism in the variables of the equation, or in the state and control
variables, with a triangular structure that induces a diffeomorphism (preserving time) in the state variables too. It is also
known as ``feedback equivalence''. 
\emph{Dynamic equivalence} refers to equivalence via invertible transformations in jet spaces that do not induce any
diffeomorphism in a finite number of variables, except when it coincides with static equivalence;
these transformations are also known as endogenous dynamic feedback~\cite{Mart92th,Flie-Lev-Mar-R92cras}, or Lie-Bäcklund
transformations~\cite{Ande-Ibr79,Flie-Lev-Mar-R92cras,Pome95vars}, although this terminology is more common for systems of partial
differential equations (PDEs); dynamic equivalence is also very close to absolute equivalence for Pfaffian
systems~\cite{Cart14abs,Shad90,Shad-Slu94}.

The literature on classification and invariants for static equivalence is too large to be quoted here; 
let us only recall that, as evidenced by all detailed studies and
mentioned in \cite{Tcho87}, each equivalence class (within control systems on the same manifold, or germs of control systems) is very very thin, indeed it has infinite co-dimension except in trivial cases. 
Since dynamic equivalence is a priori more general, it is natural to ask how more general it is.
Systems on manifolds of different dimension may be dynamic equivalent, but not static equivalent. 
Restricting our attention to systems on the same manifold and considering dynamic equivalence instead of static, how bigger are the equivalence classes~?

The literature on dynamic feedback linearization~\cite{Isid-Moo-DeL86,Char-Lev-Mar91}, differential
flatness~\cite{Flie-Lev-Mar-R92cras,Mart92th}, or absolute equivalence~\cite{Shad90} tends to describe the classes containing
linear controllable systems or ``trivial'' systems.
The authors of \cite{Flie-Lev-Mar-R92cras,Mart92th,Shad90} made the link with deep differential geometric questions dating back
to~\cite{Gour05,Cart14abs,Hilb12}; see \cite{Avan05th} for a recent overview. Despite these efforts, no
characterization is available except for systems with one control, \textit{i.e.}\ 
whose general solution depends on one function of one variable; there are many systems that one suspects to be non-flat
--\textit{i.e.} dynamic equivalent to no trivial system-- while no proof is available, see the remark on (\ref{eq:exx2}) in Section~\ref{sec-flat}.
There is however one powerful necessary condition \cite{Rouc94,Slui93}: a flat
system must be ruled, \textit{i.e.}\ its equations must define a ruled submanifold in each tangent space.
As pointed out in \cite{Rouc94}, this proves that the equivalence class of linear systems for dynamic equivalence, although bigger
than for static equivalence, still has infinite co-dimension.

Deciding whether two general systems are dynamic equivalent is at least as difficult. There is no method to prove
that two systems are not dynamic equivalent.  
The contribution of this paper is a necessary condition for two systems to be dynamic equivalent, that generalizes
\cite{Rouc94,Slui93}: if they live on manifolds of the same dimension, they must be either both ruled or static equivalent; if
not, the one of higher dimension must be ruled. 
Besides being useful to prove that some pairs of systems are not dynamic equivalent, it also implies that
``generic'' equivalence classes for dynamic equivalence are the same as for static equivalence.

\paragraph{Outline}
Notations on jet bundles and differential operators are recalled in Section~\ref{sec-notations}; the notions of systems,
ruled systems, dynamic and static equivalence are precisely defined in Section~\ref{sec-defs}.
Our main result is stated and commented in Section~\ref{sec-mainth}, and proved in Section~\ref{sec-proofs}.

\section{Miscellaneous notations}
\label{sec-notations}
Let $M$ be an $n$-dimensional manifold, either \CCC{\infty} (infinitely differentiable) or \CCC{\omega} (real analytic).
\subsection{Jet bundles}
Using the notations and definitions of \cite[Chapter II, \S2]{Golu-Gui73}, $J^k(\RR,M)$ denotes the $k$\textsuperscript{th} jet
bundle of maps $\RR\to{M}$.  It is a bundle both over $\RR$ and over $M$.  If $(x^1,\ldots,x^n)$ is a system of coordinates on an
open subset of $M$, coordinates on the lift of this open subset are given by $t,x^1,\ldots,x^n,\dot{x}^1,\ldots,\dot
x^n,\,\cdots\,$, $(x^1)^{(k)},\ldots,(x^n)^{(k)}$ where $t$ is the projection on $\RR$.

As an additive group, $\RR$ acts on $J^k(\RR,M)$ by translation of the $t$-component; the quotient by this action is well defined
and we denote it by
\begin{equation}
  \label{eq:Jk}
  J^k(M)\ =\ \left.J^k(\RR,M)\,\right/\,\RR\ .
\end{equation}
Since we only study time-invariant systems, we prefer to work with $J^k(M)$.  Quotienting indeed drops the $t$ information: local
coordinates on $J^k(M)$ are given by $x^1,\ldots,x^n,\dot{x}^1,\ldots,\dot x^n,\,\cdots\,,(x^1)^{(k)},\ldots,(x^n)^{(k)}$; for
short, we write $x,\dot{x},\ldots,x^{(k)}$.  For $\ell<k$, there is a canonical projection
\begin{equation}
  \label{eq:pikl}
  \pi_{k,\ell}:\;J^k(M)\to J^\ell(M)
\end{equation}
that makes $J^k(M)$ a bundle over $J^\ell(M)$; in particular it is a bundle over $M=J^0(M)$ and over $\mathrm{T}M=J^1(M)$. In
coordinates,
\begin{displaymath}
  \pi_{k,\ell}(x,\dot{x},\ldots,x^{(\ell)},\ldots,x^{(k)})=(x,\dot{x},\ldots,x^{(\ell)})\ .
\end{displaymath}

\paragraph*{Notation}
To a subset $\ouv\subset J^k(M)$, we associate, for all $\ell$, a subset $\ouv_\ell\subset J^\ell(M)$ in the following manner (obviously, $\ouv_k=\ouv$):
\begin{equation}
  \label{eq:0zero1}
  \ouv_\ell=\left\{
    \begin{array}{ll}
      \pi_{k,\ell}(\ouv)&\mbox{if }\ell\leq k,\\{\pi_{\ell,k}}^{-1}(\ouv)&\mbox{if }\ell\geq k.
    \end{array}
  \right.
\end{equation}

\subsection{\boldmath The $k$\textsuperscript{th} jet of a smooth (\CCC{\infty}) map $x(.):I\to M$}
\label{sec-kthjet}
With $I\subset\RR$ a time interval, it is a smooth map
$j^k(\,x(.)\,):\,I\to J^k(M)$
(see again \cite{Golu-Gui73}); in
coordinates,
\begin{displaymath}
  j^k(\,x(.)\,)\,(t)\ =\ (x(t),\dot{x}(t),\ddot{x}(t),\ldots,x^{(k)}(t))\ .
\end{displaymath}

By \emph{a smooth map whose $k$\textsuperscript{th} jet remains in $\ouv$}, for some $\ouv\subset J^k(M)$, we mean a smooth
$x(.):I\to M$ such that $j^k(\,x(.)\,)(t)\in \ouv$ for all $t$ in $I$.

\subsection{Differential operators}
\label{sec-diffop}
If $\ouv$ is an open subset of $J^k(M)$, and $M'$ is a manifold of dimension $n'$, a smooth (\CCC{\infty} or \CCC{\omega}) map
$\Phi:\,\ouv\to M'$
defines the smooth differential operator of order\footnote{
  ``Of order no larger than $k$'' would be more accurate: if $\Phi$ does not depend on $k$\textsuperscript{th} derivatives, the
  order in the usual sense would be smaller than $k$. See for instance $\Psi$ in example (\ref{eq:exx21-Psi}).} 
$k$ \vspace{-0.6em}
\begin{equation}
  \label{eq:Dphi}
  \mathcal{D}_\Phi^k\ =\ \Phi\circ j^k\ .
\end{equation}
Obviously, $\mathcal{D}_\Phi^k$ sends smooth maps $I\to M$ whose $k$\textsuperscript{th} jet remains in $\ouv$ to smooth maps
$I\to M'$.  In coordinates, the image of $t\mapsto x(t)$ is $t\mapsto\Phi(x(t),\dot{x}(t),\ddot{x}(t),\ldots,x^{(k)}(t))$. Note
that we do not require that $k$ be minimal, so $\Phi$ \emph{might} not depend on $x^{(k)}$

We call $j^r\circ\mathcal{D}_\Phi^k$
the \emph{$r$\textsuperscript{th} prolongation of} the differential operator $\mathcal{D}_\Phi^k$; it sends smooth maps $I\to M$
whose $k$\textsuperscript{th} jet remains in $\ouv$ to smooth maps $I\to J^r(M')$; it is indeed  
the differential operator $\mathcal{D}_{\Phi^{[r]}}^{k+r}$, of order $k+r$,
with $\Phi^{[r]}$ the unique smooth map ${\pi_{k+r,k}}^{-1}(\ouv)\to J^r(M')$ such that
\begin{equation}
  \label{eq:Phiprol}
  j^r\circ\Phi\circ j^k\  =\ \Phi^{[r]}\circ j^{k+r}\ .
\end{equation}
We call $\Phi^{[r]}$ the $r$\textsuperscript{th} prolongation of $\Phi$.  One has
$\pi_{r,0}\circ\Phi^{[r]}=\Phi\circ\pi_{k+r,k}$ and more generally, for $s<r$,
\begin{equation}
  \label{eq:piPhi}
  \pi_{r,s}\circ\Phi^{[r]}=\Phi^{[s]}\circ\pi_{k+r,k+s}\ .
\end{equation}

\section{Systems and equivalence}
\label{sec-defs}

\subsection{Systems}
\label{sec-sys}
\begin{definition}
  \label{def-sys}
  A \CCC{\infty} or \CCC{\omega} \emph{regular system} with $m$ controls on a smooth manifold $M$ is a \CCC{\infty} or \CCC{\omega} sub-bundle $\Sigma$ of the tangent bundle
  $\mathrm{T}M$
  \begin{equation}
    \label{diagcom}
    \begin{array}{ccc}
      \Sigma & \stackrel{i}\hookrightarrow & \mathrm{T}M
      \\
      & \!\!\pi\!\searrow & \downarrow
      \\
      & & M
    \end{array}
  \end{equation}
  with fiber $\Upsilon$, a \CCC{\infty} or \CCC{\omega} manifold of dimension $m$ (\textit{e.g.} an open subset of $\RR^m$). The
  \emph{velocity set} at a point $x\in M$ is the fiber $\Sigma_x=\pi^{-1}(\{x\})$, a submanifold of $\mathrm{T}_x M$ diffeomorphic
  to $\Upsilon$. 
\end{definition}

\begin{definition}[Solutions of a system]
  \label{def-sol}
  A solution of system $\Sigma$ on the real interval $I$ is a \emph{smooth} (\CCC{\infty}) $x(.):I\to M$ such that $j^1(x(.))(t)\in\Sigma$ for
  all $t\in I$.
\end{definition}

Although a general solution of a system need not be smooth, we only consider \emph{smooth} solutions. They form a rich enough
class in the sense that systems are fully characterized by their set of smooth solutions.

\smallskip

Locally, one may write ``explicit'' equations of $\Sigma$ in the following form.  Of course there are many choices of coordinates
and the map $f$ depends on this choice.
\begin{proposition}
  \label{lem-locexpl}
  For each $\xi\in\Sigma$, with $\Sigma\hookrightarrow\mathrm{T}M$ a regular system (\ref{diagcom}), there is
  \begin{itemize}
  \item an open neighborhood $\ouvxiBig$ of $\xi$ in $\mathrm{T}M$, $\ouvxiBig_0$ its projection on $M$,
  \item a system of local coordinates $(x_{\mathrm{I}},x_{\mathrm{I\!I}})$ on $\ouvxiBig_0$, with $x_{\mathrm{I}}$ a block of
    dimension $n-m$ and $x_{\mathrm{I\!I}}$ of dimension $m$,
  \item an open subset $\ouvgraph$ of $\RR^{n+m}$ and a smooth (\CCC{\infty} or \CCC{\omega}) map $f:\ouvgraph\to\RR^{n-m}$,
  \end{itemize}
  such that the equation of $\Sigma\cap\ouvxiBig$ in these coordinates is
  \begin{equation}
    \label{eq:eqSigma}
    \dot{x}_{\mathrm{I}}=f(x_{\mathrm{I}},x_{\mathrm{I\!I}},\dot{x}_{\mathrm{I\!I}})\,,\ \ 
    (x_{\mathrm{I}},x_{\mathrm{I\!I}},\dot{x}_{\mathrm{I\!I}})\in\ouvgraph\ .\vspace{-1em}
  \end{equation}
\end{proposition}
\begin{proof} 
  Consequence of the implicit function theorem.
\end{proof}

\paragraph{Control systems} A more usual representation of a system with $m$ controls is
\begin{equation}
  \label{eq:sysexplseul}
  \dot{x}=F(x,u)\;,\ \ \ x\in M\,,\ u\in\mathcal{B}
  \,,
\end{equation}
with $\mathcal{B}$ an open subset of $\RR^m$ and $F:M\times\mathcal{B}\to\mathrm{T}M$ smooth enough.
It can be brought locally, in block coordinates $(x_{\mathrm{I}},x_{\mathrm{I\!I}})$, to the form
\begin{equation}
\label{rev1}
  \dot{x}_{\mathrm{I}}=f(x_{\mathrm{I}},x_{\mathrm{I\!I}},u)\;,\ \ \dot{x}_{\mathrm{I\!I}}=u
\end{equation}
modulo a static feedback on $u$, at least around nonsingular points $(x,u)$ where
\begin{equation}
  \label{eq:rankexplseul}
  \rank\frac{\partial F}{\partial u}(x,u)=m\ .
\end{equation}
Equation (\ref{eq:eqSigma}) can be obtained by eliminating the control $u$ in (\ref{rev1}).

If (\ref{eq:rankexplseul}) holds, (\ref{eq:sysexplseul}) defines a system in the sense of Definition~\ref{def-sys}. 
All results on systems in that sense may easily be translated to control systems (\ref{eq:sysexplseul}).

\paragraph{Implicit systems of ODEs} A smooth \emph{system of $n-m$ ODEs} on $M$:
$R(x,\dot{x})=0$
with $R:\mathrm{T}M\to\RR^{n-m}$ also defines a system in the sense of Definition~\ref{def-sys} if it is nonsingular,
\textit{i.e.}
$\rank\frac{\partial R}{\partial\dot{x}}(x,\dot{x})=n-m$.

\paragraph{Singularities} 
With the above rank assumptions, or the one that $\Sigma$ is a sub-bundle in Definition~\ref{def-sys}, we carefully avoid singular
systems. This paper does not apply to singular control systems or singular implicit systems of ODEs.

\subsubsection*{Prolongations of $\Sigma$}
\label{sec-prolSigma}

For integers $k\geq 1$, we denote by $\prolto{\Sigma}{k}$ the prolongation of the system $\Sigma$ to $k^{th}$ order; it is
\emph{the} subbundle $\prolto{\Sigma}{k}\hookrightarrow J^k(M)$ with the following property: for any smooth map $x(.):I\to M$, 
with $j^k(x(.))$ defined in section~\ref{sec-kthjet},
\begin{equation}
  \label{eq:f(j)}
  j^1(x(.))(t)\in\Sigma\,,\ t\in I\ \ \ \ \Leftrightarrow\ \ \ \ 
  j^k(x(.))(t)\in\prolto{\Sigma}{k}\,,\ t\in I\,.
\end{equation}
The left-hand side means that $x(.)$ is a \emph{solution} of
$\Sigma$ according to Definition~\ref{def-sol}.  Obviously, $\prolto{\Sigma}{1}=\Sigma$.  We may describe $\prolto{\Sigma}{k}$ in coordinates.
\begin{proposition}
  \label{prop-prol}
  Let $K$ be a positive integer.  There is a unique sub-bundle $\prolto{\Sigma}{K}\hookrightarrow J^K(M)$ such that:
\vspace{-0.4em}
  \begin{equation}
    \label{eq:proleq}
    \begin{array}{l}
      \mbox{a smooth map $x(.):I\to M$ is
        a solution of system $\Sigma$ on the real interval $I$}\\\mbox{if and only if $j^K(x(.))(t)\in\prolto{\Sigma}{K}$ for all $t\in I$.}
    \end{array}
  \end{equation}
  For all $\xi\in\prolto{\Sigma}{K}$, its projection $\xi_1=\pi_{K,1}(\xi)$ is in $\Sigma$ and, with $\ouvxiBig$ the neighborhood
  of $\xi_1$, $(x_{\mathrm{I}},x_{\mathrm{I\!I}})$ the coordinates on $\ouvxiBig_0$, $\ouvgraph$ the open subset $\RR^{n+m}$ and
  $f:\ouvgraph\to\RR^m$ the map given by Proposition~\ref{lem-locexpl}, 
  the equations of $\ \ouvxiBig_K\cap\prolto{\Sigma}{K}\ $ in $J^K(M)$ are, in the coordinates
  $(x_{\mathrm{I}},x_{\mathrm{I\!I}},\dot{x}_{\mathrm{I}},\dot{x}_{\mathrm{I\!I}},\ldots,x^{(K)}_{\mathrm{I}},x^{(K)}_{\mathrm{I\!I}})$
  induced on $\ouvxiBig$ by $(x_{\mathrm{I}},x_{\mathrm{I\!I}})$,
  \begin{equation}
    \label{eq:eqprolong}
    \begin{array}{l}
      x_{\mathrm{I}}^{(i)}=f^{(i-1)}(x_{\mathrm{I}},x_{\mathrm{I\!I}},\dot{x}_{\mathrm{I\!I}},\ldots,x_{\mathrm{I\!I}}^{(i)})\,,\ \ \
      1\leq i\leq K\,,
      \\
      (x_{\mathrm{I}},x_{\mathrm{I\!I}},\dot{x}_{\mathrm{I\!I}},\ldots,x_{\mathrm{I\!I}}^{(K)})\;\in\;\ouvgraph\times\RR^{(K-1)m}\,,
    \end{array}
  \end{equation}
  where, for a smooth map $f:\ouvgraph\to\RR^{n-m}$, and $\ell\geq0$, $f^{(\ell)}$ is the smooth map
  $\ouvgraph\times\RR^{Km}\to\RR^{n-m}$ defined by $f^{(0)}=f$ and, for $i\geq1$,
  \begin{equation}
    \label{eq:eqprolf}
    f^{(i)}(x_{\mathrm{I}},x_{\mathrm{I\!I}},\dot{x}_{\mathrm{I\!I}},\ldots,x_{\mathrm{I\!I}}^{(i+1)})= \frac{\partial
      f^{(i-1)}}{\partial x_{\mathrm{I}}} f(x_{\mathrm{I}},x_{\mathrm{I\!I}},\dot{x}_{\mathrm{I\!I}})\,+ \sum_{i=0}^{i} \frac{\partial
      f^{(i-1)}}{\partial x_{\mathrm{I\!I}}^{(i)}} x_{\mathrm{I\!I}}^{(i+1)}\ .
  \end{equation}
\end{proposition}
\begin{proof}
  This is classical, and obvious in coordinates.
\end{proof}

\begin{remark}
  \label{rem-prol}
  {\upshape Each $\Sigma_{k+1}$ ($k\geq1$) is an affine bundle over $\Sigma_{k}$, and may be viewed as an affine
    sub-bundle of $\mathrm{T}\Sigma_k$, \textit{i.e.} it is a system in the sense of Section~\ref{sec-sys} on the manifold
    $\Sigma_{k}$ instead of $M$.

    In particular $\Sigma_2\hookrightarrow\mathrm{T}\Sigma$ is the system obtained by ``adding an integrator in each control'' of
    the system $\Sigma\hookrightarrow\mathrm{T}M$. It is an affine system (i.e. affine sub-bundle) even when $\Sigma$ is not.   }
\end{remark}

\subsection{Ruled systems}
Recall that a smooth submanifold of an affine space is ruled if and only if it is a union of straight lines,
i.e. if through each point of the submanifold passes a straight line contained in the submanifold. Such a manifold must be
unbounded; since we want to consider the intersection of a submanifold with an
arbitrary open set and allow this patch to be ``ruled'', we use the same slightly abusive notion as \cite{Land99}: a submanifold
$N$ is ruled if and only if, through each point of it, passes a straight line which is contained in $N$ ``until it reaches the
boundary of $N$''. Here, the boundary of the submanifold $N$ is $\partial N=\overline{N}\setminus N$.

A system will be called ruled if and only if $\Sigma_x$ is, for all $x$, a ruled submanifold of $\mathrm{T}_xM$. This is formalized below
in a self-contained manner.
\begin{definition}
  \label{def-ruled}
  Let $\ouvTM$ be an open subset of $\mathrm{T}M$. System $\Sigma$ (see (\ref{diagcom})) is \emph{ruled in $\ouvTM$} if and only
  if, for all $(x,\dot{x})\in\left(\ouvTM\cap\Sigma\right)$, there is a nonzero vector $w\in\mathrm{T}_x M\setminus\{0\}$ and two
  possibly infinite numbers $\lambda^-\in[-\infty,0 
  )$ and $\lambda^+\in(0,+\infty 
  ]$ such that \\$(x,\dot{x}+\lambda w)\in\ouvTM\cap\Sigma$ for all $\lambda$, $\lambda^-<\lambda<\lambda^+$ and
  \begin{equation}
    \label{eq:intmax}
    \begin{array}{l}
      \lambda^->-\infty \Rightarrow 
      (x,\dot{x}+\lambda^- w)\in\partial\left(\ouvTM\cap\Sigma\right) 
      \,,\\
      \lambda^+<+\infty \Rightarrow 
      (x,\dot{x}+\lambda^+ w)\in\partial\left(\ouvTM\cap\Sigma\right) 
      \,.
    \end{array}
  \end{equation}
Recall that, by definition, $\partial\left(\ouvTM\cap\Sigma\right)=\overline{\ouvTM\cap\Sigma}\,\setminus\left(\ouvTM\cap\Sigma\right)$.
\end{definition}

\medskip

We shall need the following characterisation.
\begin{proposition}[\cite{Land99}]
  \label{prop-landsberg}
  Let $\ouvTM$ be an open subset of $\mathrm{T}M$.  $\Sigma$ is ruled in $\ouvTM$ if and only if, for all $\xi=(x,\dot{x})$ in
  $\Sigma\cap\ouvTM$, there is a straight line in $\mathrm{T}_x M$ passing through $\dot{x}$ that has contact of infinite order
  with $\Sigma_x$ at $\dot{x}$.
\end{proposition}
\begin{proof}
  From \cite[Theorem 1]{Land99}, a ``patch of'' submanifold of dimension $m$ in a manifold of dimension $n$ is
  ruled if and only if there is, through each point, a straight line that has contact of order $n+1$. This is of course implied by
  infinite order.
\end{proof}

\subsection{Dynamic equivalence}
The following notion is usually called dynamic equivalence, or equivalence by (endogenous) dynamic feedback transformations in
control theory, see \cite{Mart92th,Flie-Lev-Mar-R99geo,Jaku94,Pome95vars}. It is in fact also the notion of
Lie-Bäcklund transformation, limited to ordinary differential equation, as noted in
\cite{Flie-Lev-Mar-R99geo} or \cite{Pome95vars}.

\begin{definition}
  \label{def-eq}
  Let $\Sigma\hookrightarrow\mathrm{T}M$ and $\Sigma'\hookrightarrow\mathrm{T}M'$ be \CCC{\infty} (resp. \CCC{\omega}) regular systems
  (see (\ref{diagcom})) on two manifolds $M$ and $M'$ of dimension $n$ and $n'$, $K,K'$ two integers, $\ouv\subset
  J^{K}(M)$ and $\ouv'\subset J^{K'}(M')$ two open subsets.

  Systems $\Sigma$ and $\Sigma'$ are \emph{dynamic equivalent over $\ouv$ and $\ouv'$} if and only if there exists two mappings of
  class \CCC{\infty} (resp. \CCC{\omega})~:\vspace{-0.5em}
  \begin{equation}
    \label{eq:phipsi}
    \Phi:\,\ouv\to M'
    \,,\ \ \ \ 
    \Psi:\,\ouv' \to M\vspace{-0.5em}
  \end{equation}
  inducing differential operators $\mathcal{D}_\Phi^K$ and $\mathcal{D}_\Psi^{K'}$ --see (\ref{eq:Dphi})-- such that, for any
  interval $I$,
  \begin{itemize}
  \item for any solution $x(.):I\to M$ of $\Sigma$ whose $K$\textsuperscript{th} jet remains inside $\ouv$, \\
    $\mathcal{D}_\Phi^{K}(\,x(.)\,)$ is a solution of $\Sigma'$ whose $K'$\textsuperscript{th} jet remains inside $\ouv'$ \\ and
    $\mathcal{D}_\Psi^{K'}(\,\mathcal{D}_\Phi^K(\,x(.)\,)\,)\ =\ x(.)$,
  \item for any solution $z(.):I\to M'$ of $\Sigma'$ whose $K'$\textsuperscript{th} jet remains inside $\ouv'$, \\
    $\mathcal{D}_\Psi^{K'}(\,z(.)\,)$ is a solution of $\Sigma$ whose $K$\textsuperscript{th} jet remains inside $\ouv$ \\ and
    $\mathcal{D}_\Phi^K(\,\mathcal{D}_\Psi^{K'}(\,z(.)\,)\,)\ =\ z(.)$.
  \end{itemize}
\end{definition}

\medskip

\begin{remark}
  \label{rem-intext}
  {\upshape Since all properties are tested on \emph{solutions}, only the restriction of $\Phi$ and $\Psi$ to $\Sigma_K$ and
    $\Sigma_{K'}$ (see Proposition~\ref{prop-prol}) matter; for instance, $\Phi$ can be arbitrarily modified away from $\Sigma_K$
    without changing any conclusions. Borrowing this language from the literature on Lie-B\"{a}cklund transformations, $\Phi$ and
    $\Psi$ above are ``external'' correspondences.

    In \cite{Flie-Lev-Mar-R99geo} or in \cite{Pome95vars}, the ``internal'' point of view prevails: for instance $\Phi$ and $\Psi$
    are replaced, in \cite{Flie-Lev-Mar-R99geo}, by diffeomorphisms between diffieties. This is more intrinsic because maps are
    defined only where they are to be used. However the definitions are equivalent because these internal maps admit infinitely
    many ``external'' prolongations.

    Here, this external point of view is adopted because it makes the statement of the main result less technical. Note however
    that, as a preliminary to the proofs, an ``internal'' translation is given in section~\ref{sec-altdef}. }
\end{remark}

\begin{remark}
  \label{rem-goodOmeg}
  {\upshape In the theorems, we shall require that $\ouv$ and $\ouv'$ satisfy\vspace{-0.4em}
    \begin{equation}\vspace{-0.4em}
      \label{eq:condOmega}
      \ouv_1\cap\Sigma\;\subset\;\left(\ouv\cap\prolto{\Sigma}{K}\right)_1
      \ \ \mbox{and}\ \ \ 
      \ouv'_1\cap\Sigma'\;\subset\;\left(\ouv'\cap\prolto{\Sigma'}{K'}\right)_1\ ,
    \end{equation}
    \textit{i.e.} any (jet of) solution whose first jet is in $\ouv_1$ lifts to at least one (jet of) solution whose
    $K$\textsuperscript{th} jet is in $\ouv$.  Note the following facts about this requirement.  \\- These inclusions are
    equalities for the reverse inclusions always hold.  \\- Replacing the original $\ouv$ with
    $\ouv\setminus\left(\overline{(\ouv_1\cap\Sigma)\setminus(\ouv\cap\prolto{\Sigma}{K})_1}\right)_K$ and $\ouv'$ accordingly
    forces (\ref{eq:condOmega}); alternatively, keeping arbitrary open sets, Theorem~\ref{th-main} and Theorem~\ref{th-flat} would
    hold with $\Omega_1$ replaced with $\Omega_1\setminus\overline{(\ouv_1\cap\Sigma)\setminus(\ouv\cap\prolto{\Sigma}{K})_1}$.
    \\- When $\Sigma'=\mathrm{T}M'$ is the trivial system (see section~\ref{sec-ex}), any open $\ouv'$ satisfies
    (\ref{eq:condOmega}).  }
\end{remark}

\subsection{Static equivalence}
\begin{definition}
  \label{def-static}
  Let $\ouvTM\subset\mathrm{T}M$ and $\ouvTM'\subset\mathrm{T}M'$ be open subsets.  Systems $\Sigma$ and $\Sigma'$ are
  \emph{static equivalent} over $\ouvTM$ and $\ouvTM'$ if and only if there is a smooth diffeomorphism $\Phi:\ouvTM_0\to\ouvTM_0'$
  such that the following holds: 
\vspace{-0.5em}
\begin{equation}
    \label{eq:stateq}
    \vspace{-0.6em}
    \left.\!\!
      \begin{array}{l}
        \mbox{a smooth map $t\mapsto x(t)$ is a solution of $\Sigma$ whose first jet remains in $\ouvTM$} \\
        \mbox{if and only if $t\mapsto\Phi(x(t))$ is a solution of $\Sigma'$ whose first jet remains in $\ouvTM'$.}
      \end{array}
    \right\}
  \end{equation}
\end{definition}
\begin{definition}[Local static equivalence]
  \label{def-locstat}  Let $\ouvTM\subset\mathrm{T}M$ and $\ouvTM'\subset\mathrm{T}M'$ be open subsets.  Systems $\Sigma$ and
  $\Sigma'$ are \emph{locally} static equivalent over $\ouvTM$ and $\ouvTM'$ if and only if there are coverings of $\ouvTM\cap\Sigma$ and $\ouvTM'\cap\Sigma'$~:\vspace{-0.5em}
      \begin{equation*}\vspace{-0.5em}
        \Sigma\cap\ouvTM\ \subset\ \Sigma\cap\bigcup_{\alpha\in A}\ouvTM^\alpha\,,\ \
        \Sigma'\cap\ouvTM'\ \subset\ \Sigma'\cap\bigcup_{\alpha\in A}\ouvTM^{\prime\alpha}
      \end{equation*}
      where $A$ is a set of indices, $\ouvTM^\alpha$ and $\ouvTM^{\prime\alpha}$ are open subsets of $\ouvTM$ and
      $\ouvTM'$, such that, for all $\alpha$, systems $\Sigma$ and $\Sigma'$ are static equivalent over $\ouvTM^\alpha$ and
      $\ouvTM^{\prime\alpha}$. 
\end{definition}

This definition, stated in terms of solutions, is translated into point (a) below, that only relies on the geometry of $\Sigma$ and
$\Sigma'$ as submanifolds. Point (b) is used for instance in \cite{Kupk91,Wilk99} where ``centro-affine''
geometry of each $\Sigma_x$ is studied.
\begin{proposition}
  \label{prop-static}
  \quad\textup{(a)}\ Systems $\Sigma$ and $\Sigma'$ are \emph{static equivalent} over $\ouvTM\subset\mathrm{T}M$ and
  $\ouvTM'\subset\mathrm{T}M'$ if and only there is a smooth diffeomorphism $\Phi:\ouvTM_0\to\ouvTM_0'$ such that ${\Phi}_\star$
  maps $\ouvTM\cap\Sigma$ to $\ouvTM'\cap\Sigma'$.

  \textup{(b)}\ If systems $\Sigma$ and $\Sigma'$ are static equivalent over $\ouvTM\subset\mathrm{T}M$ and
  $\ouvTM'\subset\mathrm{T}M'$, there is, for each $x\in \ouvTM_0$ a \emph{linear isomorphism} $\mathrm{T}_x
  M\to\mathrm{T}_{\Phi(x)}M'$ that maps $\Sigma_x$ to $\Sigma'_{\Phi(x)}$.

  \vspace{-.2ex}
  \textup{(c)}\ Static equivalence preserves ruled systems.
\end{proposition}
\begin{proof}
  (b) and (c) are easy consequences of (a), which in turn is clear by differentiating solutions in Definition~\ref{def-sol}.
\end{proof}

\subsection{Examples}
\label{sec-ex}

\quad\textit{1.}
We call \emph{trivial system} on a smooth manifold $M$ the tangent bundle itself $\mathrm{T}M$. Any smooth $x(.):I\to M$
is a solution of this system; it corresponds to ``no equation'', or to the control system $\dot x=u$, or
to the ``affine diffieties'' in \cite{Flie-Lev-Mar-R99geo}.
Following \cite{Flie-Lev-Mar-R92cras,Flie-Lev-Mar-R99geo}, a system $\Sigma\hookrightarrow\mathrm{T}M$ is called
\emph{differentially flat} (on $\Omega\subset J^K(M)$) if and only if it is dynamic equivalent (over $\ouv$ and $\ouv'$) to the
trivial system $\mathrm{T}M'$ for some manifold $M'$. 

\smallskip

\paragraph{2}
Any system $\Sigma\hookrightarrow\mathrm{T}M$ is dynamic equivalent to the one
obtained by ``adding integrators''. 
It was described in Remark~\ref{rem-prol} as an affine sub-bundle $\Sigma_2\hookrightarrow\mathrm{T}\Sigma$; $\Sigma$ and
$\Sigma_2$ are equivalent in the sense of
Definition~\ref{def-eq} with $M'=\Sigma$, $K=1$, $K'=0$, $\ouv$ an open neighborhood of $\Sigma$ in $J^1(M)=\mathrm{T}M$ such that
there is a $\Phi:\ouv\to\Sigma$ that coincides with identity on $\Sigma$, $\ouv'=M'=\Sigma$ and $\Psi=\pi$ (see (\ref{diagcom})).

This may be easier to follow in the coordinates of Proposition~\ref{lem-locexpl}. The prolongation of (\ref{eq:eqSigma}) has state 
$(y_{\mathrm{I}},y_{\mathrm{I\!I}})\in\ouvgraph$, with $y_{\mathrm{I}}$ a block of dimension $n$ and $y_{\mathrm{I\!I}}$ of
dimension $m$, and equation
$\dot{y}_{\mathrm{I}}=\bigl( f(y_{\mathrm{I}},y_{\mathrm{I\!I}})\,,\, y_{\mathrm{I\!I}}\bigr)\,$.
In coordinates, the transformations $\Phi:J^1(\ouvgraph_0)\to\ouvgraph$ and $\Psi:\ouvgraph\to\ouvgraph_0$ are given by
$(y_{\mathrm{I}},y_{\mathrm{I\!I}})=\Phi(x_{\mathrm{I}},x_{\mathrm{I\!I}},\dot{x}_{\mathrm{I}},\dot{x}_{\mathrm{I\!I}})=(x,\dot{x}_{\mathrm{I\!I}})$
and $x=\Psi(y)=y_{\mathrm{I}}$.

Static equivalence between these systems of different dimension does not hold.

\smallskip

\paragraph{3}
Let us now give, mostly to illustrate the role of the integers $K,K'$ and the open sets $\ouv$ and $\ouv'$, two more specific
examples of systems $\Sigma\hookrightarrow\mathrm{T}\RR^3$ and $\Sigma'\hookrightarrow\mathrm{T}\RR^3$ with the following
equations in $\mathrm{T}\RR^3$, with coordinates
$(x_1,x_2,x_3,\dot{x}_1,\dot{x}_2,\dot{x}_3)$ or $(y_1,y_2,y_3,\dot{y}_1,\dot{y}_2,\dot{y}_3)$, clearly defining sub-bundles
with fiber diffeomorphic to $\RR^2$:
\begin{equation}
  \label{eq:exx1}
  \Sigma:\ \dot{x}_1=x_2\;,\ \ \ \ \Sigma':\ \dot{y}_1=y_2+(\dot{y}_2-y_1{\dot{y}_3})\,\dot{y}_3\ .
\end{equation}
These equations are even globally in the ``explicit'' form given by Proposition~\ref{lem-locexpl}.

\smallskip

First of all, $\Sigma$ is dynamic equivalent to the trivial system $\Sigma''=\mathrm{T}\RR^2$, with $\Phi:\RR^3\to\RR^2$ defined
by $\Phi(x_1,x_2,x_3)=(x_1,x_3)$ and $\Psi:J^1(\RR^2)\to\RR^3$ given by $\Psi(z_1,z_2,\dot{z}_1,\dot{z}_2)=(z_1,\dot{z}_1,z_2)$.
Here $K=0, K'=1, \ouv=\RR^2, \ouv'=J^1(\RR^2)$.

\smallskip

Also, with $K=1$ and $K'=2$, systems \emph{$\Sigma$ and $\Sigma'$ are dynamic equivalent} over $\ouv\subset J^1(\RR^3)$ and $\ouv'\subset
J^2(\RR^3)$ defined by 
\begin{equation*}
  \begin{split}
    \ouv= \{&(x_1,x_2,x_3,\,\dot{x}_1,\dot{x}_2,\dot{x}_3),\ 1-\dot{x}_2-{x_2}^3\neq0\}\,,
    \\
    \ouv'= \{&(y_1,y_2,y_3,\,\dot{y}_1,\dot{y}_2,\dot{y}_3,\,\ddot{y}_1,\ddot{y}_2,\ddot{y}_3),\ 1-\ddot{y}_3-{\dot{y}_3}^3\neq0\}\,.
  \end{split}
\end{equation*}
The maps $\Phi:\ouv\to\RR^3$ and $\Psi:\ouv'\to\RR^3$ are given by
\begin{gather}
  \label{eq:exx12-Phi}
  \Phi(x_1,x_2,x_3,\,\dot{x}_1,\dot{x}_2,\dot{x}_3)= (\,\frac{(1-\dot{x}_2)x_3+x_2\,\dot{x}_3}{1-\dot{x}_2-{x_2}^3}
  \,,\,\frac{{x_2}^2\,x_3+\dot{x}_3}{1-\dot{x}_2-{x_2}^3} \,,\,x_1\,)\,,
  \\
  \label{eq:exx21-Psi}
    \Psi(y_1,y_2,y_3,\,\dot{y}_1,\dot{y}_2,\dot{y}_3,\,\ddot{y}_1,\ddot{y}_2,\ddot{y}_3)=
  (\,y_3\,,\,\dot{y}_3\,,\,y_1-\dot{y}_3\,y_2\,)\,.
\end{gather}

\begin{remark}
  \label{rem-foot}
  {\upshape Since $\Psi$ does not depend on second derivatives, $K'=2$ is not the order of the differential operator
$\mathcal{D}^{K'}_{\Psi}$ in the usual sense; this illustrates the footnote after (\ref{eq:Dphi}); it is however necessary to go
to second jets to describe the domain $\ouv'$ where the restriction to solutions of $\Sigma'$ of this first order operator can be inverted.}
\end{remark}

\smallskip

Finally, note that systems \emph{$\Sigma$ and $\Sigma'$ are not static equivalent} because, from
Proposition~\ref{prop-static}-(b), this would imply that each $\Sigma_x$ is sent to some $\Sigma'_y$ by a linear isomorphism
$\mathrm{T}_xM\to\mathrm{T}_yM'$, which is not possible because each $\Sigma_x$ is an affine subspace of $\mathrm{T}_x M$ and
$\Sigma'_y$ a non degenerate quadric of $\mathrm{T}_y M'$.

\smallskip

\paragraph{4}
Consider two more systems, $\Sigma\hookrightarrow\mathrm{T}\RR^3$ and $\Sigma'\hookrightarrow\mathrm{T}\RR^3$ described as in (\ref{eq:exx1}):
\begin{equation}
  \label{eq:exx2}
  \Sigma:\ \dot{x}_1=x_2+(\dot{x}_2-x_1{\dot{x}_3})^2\,\dot{x}_3^{\;2}\;,\ \ \ \ \Sigma':\ \dot{y}_1=y_2+(\dot{y}_2-y_1{\dot{y}_3})^2\,\dot{y}_3\ .
\end{equation}
System $\Sigma$ is ruled --each $\Sigma_y$ is the union of lines $\dot{y}_2-y_1{\dot{y}_3}=\lambda$,
$\dot{y}_1=y_2+\lambda^2\,\dot{y}_3$ for $\lambda$ in $\RR$-- while $\Sigma'$ is not. 
Hence, from point (c) of Proposition~\ref{prop-static}, \emph{$\Sigma$ and $\Sigma'$ are not static equivalent}.
We shall come back to these two systems from the point of view of flatness and dynamic equivalence in sections \ref{sec-flat} and \ref{sec-main}.

\section{Necessary conditions}
\label{sec-mainth}




\subsection{The case of flatness}
\label{sec-flat}
It has been known since \cite{Rouc94,Slui93} that a system which is dynamic equivalent to a \emph{trivial system} --see the
beginning of section~\ref{sec-ex}; such a system is called differentially flat-- must be ruled; of course, at least in the smooth
case, this is true only on the domain where equivalence is assumed.

\begin{theorem}[\cite{Rouc94,Slui93}]
\label{th-flat}
  If $\Sigma$ is dynamic equivalent to the trivial system $\Sigma'\!\!=\!\mathrm{T}M'$ over $\ouv\subset J^K(M)$ and $\ouv'\subset
  J^{K'}\!(M')$ satisfying (\ref{eq:condOmega}), then $\Sigma$  is ruled in $\ouv_1$.
\end{theorem}

\paragraph{Application}
Since $\Sigma$ in (\ref{eq:exx2}) is not ruled, this theorem implies that it is not flat, i.e. not dynamic equivalent to the
trivial system $\mathrm{T}\RR^2$. On the contrary, $\Sigma'$ in (\ref{eq:exx2})
is ruled, hence the result does not help deciding it being flat or not; in fact, one conjectures that this system is not flat, but
no proof is available; see~\cite{Avan-Pom07}.

\subsection{Main idea of the proofs}
\label{sec-sketch}
Our main result, stated in next section, studies what remains of Theorem~\ref{th-flat} when $\Sigma'$ is not the trivial system.
Due to many technicalities concerning regularity conditions, the main ideas may be difficult to grasp in the proof given in
section~\ref{sec-proof}. In order to enlighten these ideas, and even the result itself, let us first sketch the proof of the above
theorem, following the line of \cite{Rouc94} (itself inspired from \cite{Hilb12}), but \emph{without assuming a priori that
$\Sigma'$ is trivial}. 

Take two arbitrary systems $\Sigma$ and $\Sigma'$, and assume that they are dynamic equivalent.
From Proposition~\ref{lem-locexpl}, one may use locally the explicit forms
\begin{displaymath}
  \Sigma:\;\dot{x}_{\mathrm{I}}=f(x_{\mathrm{I}},x_{\mathrm{I\!I}},\dot{x}_{\mathrm{I\!I}})\;,\ \ \ 
  \Sigma':\;\dot{z}_{\mathrm{I}}=g(z_{\mathrm{I}},z_{\mathrm{I\!I}},\dot{z}_{\mathrm{I\!I}})\;.
\end{displaymath}
Recall that $n$ and $n'$ denote the dimensions of $x$ and $z$; assume $n\leq n'$.
Since we work only on \emph{solutions} (see Remark~\ref{rem-intext} and also Section~\ref{sec-altdef})
and the above equations allow one to express each time-derivative $x_{\mathrm{I}}^{(j)}$, $j\geq1$, as a function of
$x_{\mathrm{I}},x_{\mathrm{I\!I}}$, $\dot{x}_{\mathrm{I\!I}}$, $\ldots$, $x_{\mathrm{I\!I}}^{(j)}$, we may work
with the variables 
$x_{\mathrm{I}},x_{\mathrm{I\!I}},\dot{x}_{\mathrm{I\!I}},\ddot{x}_{\mathrm{I\!I}},x_{\mathrm{I\!I}}^{(3)}\!\!,\ldots$ 
and $z_{\mathrm{I}},z_{\mathrm{I\!I}},\dot{z}_{\mathrm{I\!I}},\ddot{z}_{\mathrm{I\!I}},z_{\mathrm{I\!I}}^{(3)}\!\!,\ldots$ only.
The map $\Phi$ of Definition~\ref{def-eq} translates, in these coordinates, into a correspondence $z_{\mathrm{I}}=\phi_{\mathrm{I}}(x_{\mathrm{I}},x_{\mathrm{I\!I}},\dot{x}_{\mathrm{I\!I}},\ldots,x_{\mathrm{I\!I}}^{(K)})$, 
$z_{\mathrm{I\!I}}=\phi_{\mathrm{I\!I}}(x_{\mathrm{I}},x_{\mathrm{I\!I}},\dot{x}_{\mathrm{I\!I}},\ldots,x_{\mathrm{I\!I}}^{(K)})$;
here the number $K$ is chosen such that \emph{the dependence of $\phi$ versus $x_{\mathrm{I\!I}}^{(K)}$ is effective}.

If $K=0$, this reads $z=\phi(x)$, and $n<n'$ is absurd because it would imply (around points where
  the rank of $\phi$ is constant) some nontrivial relations $R(z)=0$. Hence $n=n'$, $\phi$ is a local diffeomorphism
  and static equivalence holds locally. 

If $K\geq1$, note that $\Phi$ mapping solutions of $\Sigma$ to solution
of $\Sigma'$ implies (plug the expression of $z$ given by $\phi$ into state equations of $\Sigma'$) the following identity, valid
for all $x_{\mathrm{I}},x_{\mathrm{I\!I}},\dot{x}_{\mathrm{I\!I}},\ldots,x_{\mathrm{I\!I}}^{(K+1)}$:
\begin{align*}
  &\frac{\partial\phi_{\mathrm{I}}}{\partial x_{\mathrm{I}}} f(x_{\mathrm{I}},x_{\mathrm{I\!I}},\dot{x}_{\mathrm{I\!I}})
  + \frac{\partial\phi_{\mathrm{I}}}{\partial x_{\mathrm{I\!I}}} \dot{x}_{\mathrm{I\!I}}
  + \frac{\partial\phi_{\mathrm{I}}}{\partial \dot{x}_{\mathrm{I\!I}}} \ddot{x}_{\mathrm{I\!I}}
  +\cdots+ \frac{\partial\phi_{\mathrm{I}}}{\partial x_{\mathrm{I\!I}}^{(K)}} x_{\mathrm{I\!I}}^{(K+1)}
\\
&\ \ =\ \ g\left(
\phi_{\mathrm{I}},\phi_{\mathrm{I\!I}},
\,
\frac{\partial\phi_{\mathrm{I\!I}}}{\partial x_{\mathrm{I}}} f(x_{\mathrm{I}},x_{\mathrm{I\!I}},\dot{x}_{\mathrm{I\!I}})
  + \frac{\partial\phi_{\mathrm{I\!I}}}{\partial x_{\mathrm{I\!I}}} \dot{x}_{\mathrm{I\!I}}
  + \frac{\partial\phi_{\mathrm{I\!I}}}{\partial x_{\mathrm{I\!I}}} \ddot{x}_{\mathrm{I\!I}}
  +\cdots+ \frac{\partial\phi_{\mathrm{I\!I}}}{\partial x_{\mathrm{I\!I}}^{(K)}} x_{\mathrm{I\!I}}^{(K+1)}
\,
\right)
\end{align*}
where $\phi_{\mathrm{I}}$ and $\phi_{\mathrm{I\!I}}$ depend on
$x_{\mathrm{I}},x_{\mathrm{I\!I}},\dot{x}_{\mathrm{I\!I}},\ldots,x_{\mathrm{I\!I}}^{(K)}$ only and, at least at generic points,
$(\,\frac{\partial\phi_{\mathrm{I}}}{\partial x_{\mathrm{I\!I}}^{(K)}}\, , \, \frac{\partial\phi_{\mathrm{I\!I}}}{\partial 
  x_{\mathrm{I\!I}}^{(K)}} \,) \neq (0,0)$. 
Fixing such $x_{\mathrm{I}},x_{\mathrm{I\!I}},\dot{x}_{\mathrm{I\!I}},\ldots,x_{\mathrm{I\!I}}^{(K)}$ and consequently
$z=\phi(x_{\mathrm{I}}$, $x_{\mathrm{I\!I}},\dot{x}_{\mathrm{I\!I}},\ldots,x_{\mathrm{I\!I}}^{(K)})$, and examining $\Sigma'_z$ as a
submanifold of $T_{z}M'$ with equation $\dot{z}_{\mathrm{I}}=g(z,\dot{z}_{\mathrm{I\!I}})$, it is clear that moving
$x_{\mathrm{I\!I}}^{(K+1)}$ in a direction which is not in the kernel of 
$\frac{\partial\phi_{\mathrm{I\!I}}}{\partial
  x_{\mathrm{I\!I}}^{(K)}}(x_{\mathrm{I}},x_{\mathrm{I\!I}},\dot{x}_{\mathrm{I\!I}},\ldots,x_{\mathrm{I\!I}}^{(K)})$ provides a
straight line of $T_{z}M'$ contained in $\Sigma'_z$ and, since this
covers all points of $\Sigma'_z$, proves that the latter is a ruled submanifold of $T_{z}M'$ and finally that system $\Sigma'$ is ruled.
We only examined regular points; see Section~\ref{sec-proof} for a proper proof.

Collecting the two cases, we have proved that, if $n\leq n'$, either $\Sigma'$ is ruled or $n=n'$ and $\Sigma'$ is static
equivalent to $\Sigma$. This is stated formally in Theorem~\ref{th-main}.

\subsection{The result for general systems}
\label{sec-main}
The contribution of this paper is the following strong necessary condition for 
dynamic equivalence between two general systems. $\ouv_1$ and $\ouv'_1$ are defined by (\ref{eq:0zero1}).

\smallskip

\begin{theorem}
  \label{th-main}
  Let $\Sigma$ and $\Sigma'$ be systems on manifolds of dimension $n$ and $n'$, $K,K'$ two integers and $\ouv\subset J^K(M)$, $\ouv'\subset J^{K'}(M')$ two open
  subsets satisfying (\ref{eq:condOmega}).

  If $\Sigma$ and $\Sigma'$ are dynamic equivalent over $\ouv$ and $\ouv'$, then
  \begin{description}
  \item[\boldmath if $n>n'$,] system $\Sigma$ is ruled in $\ouv_1$,
  \item[\boldmath if $n<n'$,] system $\Sigma'$ is ruled in $\ouv'_1$,
  \item[\boldmath if $n=n'$,] then \hfill (see Definition~\ref{def-locstat} for ``locally static equivalent'')
  \item[\ -] \emph{in the real analytic case}, and if $\ouv_1\cap\Sigma$ and $\ouv'_1\cap\Sigma'$ are connected, \\either systems $\Sigma$ and $\Sigma'$ are ruled in $\ouv_1$ and $\ouv'_1$ respectively, \\or
    they are locally static equivalent over $\ouv_1$ and $\ouv'_1$,
  \item[\ -] \emph{in the smooth (\CCC{\infty}) case}, there are open subsets $\mathcal{R},\mathcal{S}$ of $\ouv_1$ and $\mathcal{R}',\mathcal{S}'$ 
    of $\ouv'_1$ \\such that $\ouv_1$ and $\ouv'_1$ are covered as\vspace{-0.5em}
    \begin{equation}    \vspace{-0.5em}
      \label{eq:partition}
      \ouv_1 = \overline{\mathcal{R}}\cup\mathcal{S}= \mathcal{R}\cup\overline{\mathcal{S}}\,,\ \ \ 
\ouv'_1 = \overline{\mathcal{R}'}\cup\mathcal{S}' = \mathcal{R}'\cup\overline{\mathcal{S}'}
    \end{equation}
and the systems have the following properties on these sets:
    \begin{enumerate}
    \item $\Sigma$ and $\Sigma'$ are ruled in $\mathcal{R}$ and $\mathcal{R}'$ respectively,
    \item $\Sigma$ and $\Sigma'$ are \emph{locally} static equivalent over $\mathcal{S}$ and $\mathcal{S}'$.
    \end{enumerate}
  \end{description}
\end{theorem}
\begin{proof}
  See Section~\ref{sec-proof}.
\end{proof}

\smallskip
A few remarks are in order:

\smallskip
\paragraph{1} Theorem \ref{th-flat} is a consequence. Indeed, $n'\!\!=\!m'$ because $\Sigma'$ is trivial, dynamic equivalence implies $m'\!\!=\!m$ 
(this is common knowledge; see \cite{Cart14abs}, \cite{Flie-Lev-Mar-R99geo} or
\cite[Theorem 1]{Pome95vars}), and $n\geq m$ for any system; hence $n\geq n'$ and Theorem~\ref{th-main} directly implies that $\Sigma$ is ruled except
  if the systems are static equivalent, but this also implies that $\Sigma$ is ruled from point (c) of
  Proposition~\ref{prop-static} and the fact that the trivial system $\Sigma'$ is ruled. 

  Static equivalence still appears explicitly in Theorem~\ref{th-main} because two general systems can be static equivalent
  without being ruled.

\smallskip
\paragraph{2} The part ``$n>n'$ or $n<n'$'' can be rephrased as follows: if a system is not ruled, it cannot be dynamic equivalent to any
  system of smaller dimension. No necessary condition is given on the system of lower dimension; indeed \emph{any} system is
  dynamic equivalent to at least its first prolongation, see Example~2 in Section~\ref{sec-ex}.

\smallskip
\paragraph{3} The case $n=n'$ states that dynamic equivalence, except when it reduces to static equivalence, forces both systems to be
  ruled (in the real analytic case, the added rigidity prevents the two situations from occurring simultaneously).

  In other words, if two systems are not static equivalent and at least one of them is not ruled, they are not dynamic equivalent.
  Since the two conditions can be checked rather systematically, this yields a new and powerful method for proving that two
  systems are \emph{not} dynamic equivalent, a difficult task in general because very few invariants of dynamic equivalence are
  known. 

  For instance, to the best of our knowledge, the state of the art does not allow one to decide whether $\Sigma$ and $\Sigma'$ in
  (\ref{eq:exx2}) are dynamic equivalent or not. In section~\ref{sec-ex}, it was noted that they are not static equivalent
  and $\Sigma'$ is not ruled. This implies:
  \begin{corollary}
    \label{cor-ex}
    $\Sigma$ and $\Sigma'$ in (\ref{eq:exx2}) are not dynamic equivalent over any domains.
  \end{corollary}

\smallskip
\paragraph{4} Since being ruled is non-generic~\cite{Rouc94}, we have the following general consequence (in terms of germs of systems because the
  conclusion in the theorem is only local).
  \begin{corollary}
    \label{cor-gener}
    Generic static equivalence classes for germs of systems of the same dimension at a point are also dynamic equivalence classes.
  \end{corollary}

  Note that this is in the mathematical sense of ``generic'': this does not prevent many interesting systems from being dynamic
  equivalent without being static equivalent... it might even be that ``most interesting systems'' fall in this case~!

\section{Proofs}
\label{sec-proofs}
Recall that subscripts always refer to the order of the jet space. The notation (\ref{eq:0zero1}) is constantly used.

\subsection{Preliminaries: a re-formulation of dynamic and static equivalence}
\label{sec-altdef}
The maps $\Phi$ and $\Psi$ are always applied to jets of solutions, and, according to (\ref{eq:f(j)}), the $K$\textsuperscript{th} jets of
solutions of $\Sigma$ remain in $\prolto{\Sigma}{K}$; hence the only information to retain about $\Phi$ and $\Psi$ is their
restriction to, respectively,
\begin{equation}
  \label{eq:omegtilde}
  \widetilde{\ouv}=\ouv\cap\prolto{\Sigma}{K}\ \ \mbox{and}\ \ 
  \widetilde{\ouv}'=\ouv'\cap\prolto{\Sigma'}{K'}\,.
\end{equation}

We need one more piece of notation: according to Section~\ref{sec-diffop}, the
$\ell$\textsuperscript{th} prolongation of a smooth map $\widetilde{\Phi}:\widetilde{\ouv}\to M'$,
is a map $\pi_{K+\ell,\ell}^{\ -1}(\widetilde{\ouv})\to J^\ell{M'}$; again, only its restriction to $\prolto{\widetilde{\ouv}}{K+\ell}$ will
matter; for this reason, the notations $\widetilde{\Phi}^{[\ell]}$ and $\widetilde{\Psi}^{[\ell]}$ will not stand for the prolongations as defined earlier, but
rather these restrictions:
\begin{align}
  \label{eq:Phi[l]}
  & \widetilde{\Phi}^{[\ell]}:\prolto{\widetilde{\ouv}}{K+\ell}\to J^\ell(M') \,,\ \ \
  \widetilde{\Psi}^{[\ell]}:\prolto{\widetilde{\ouv'}}{K'+\ell}\to J^\ell(M) \,,
  \\
  \label{eq:omegtildeprol}
  \mbox{with}\hspace{1em}& \prolto{\widetilde{\ouv}}{K+\ell}=\prolto{\ouv}{K+\ell}\cap\prolto{\Sigma}{K+\ell}\,,\ \;
  \prolto{\widetilde{\ouv}'}{K'+\ell}=\prolto{\ouv'}{K'+\ell}\cap\prolto{\Sigma'}{K'+\ell}\,.
\end{align}
We may now state the following proposition. Smooth (\CCC{\infty} or \CCC{\omega}) maps on $\prolto{\widetilde{\ouv}}{K+\ell}$ or
$\prolto{\widetilde{\ouv'}}{K'+\ell}$ can be defined in a standard way because, from Proposition~\ref{lem-locexpl}, these are
smooth embedded submanifolds. 
\begin{proposition}[Dynamic Equivalence]
  \label{prop-eq}
  Let $K,K'$ be integers, $\ouv\subset J^{K}(M)$ and $\ouv'\subset J^{K'}(M')$ two open subsets. Systems $\Sigma$ and $\Sigma'$
  are \emph{dynamic equivalent over $\ouv$ and $\ouv'$} if and only if, with $\widetilde{\ouv},\widetilde{\ouv}'$ defined in
  (\ref{eq:omegtilde}), there exist two smooth (real analytic, in the real analytic case) mappings
  \begin{displaymath}
    \widetilde{\Phi}:\widetilde{\ouv}\to M'
    \ \ \mbox{and} \ \ \ 
    \widetilde{\Psi}:\widetilde{\ouv}'\to M\,,\vspace{-1em}
  \end{displaymath}
  such that\vspace{-1em}
  \begin{equation}
    \label{eq:eqgeom1}
    \widetilde{\Phi}^{[1]}(\prolto{\widetilde{\ouv}}{K+1})\subset\Sigma'\,,\ \ \widetilde{\Psi}^{[1]}(\prolto{\widetilde{\ouv'}}{K'+1})\subset\Sigma\,,
  \end{equation}
  and, with $\widetilde{\Phi}^{[K]}$ and $\widetilde{\Psi}^{[K]}$ defined by (\ref{eq:Phi[l]}),
  \begin{gather}
    \label{eq:eqgeom2}
    \widetilde{\Phi}^{[K']}(\prolto{\widetilde{\ouv}}{K+K'})\subset\ouv'\,,\
    \widetilde{\Psi}^{[K]}(\prolto{\widetilde{\ouv}'}{K+K'})\subset\ouv\,,
    \\
    \label{eq:eqgeom3}
    \widetilde{\Psi}\circ\widetilde{\Phi}^{[K']}= \left. \vphantom{\Phi^{[K]}}
      \pi_{K+K',\,0}\right|_{\prolto{\widetilde{\ouv}}{K+K'}}\,,\ \ \widetilde{\Phi}\circ\widetilde{\Psi}^{[K]}=
    \left.\vphantom{\Phi^{[K]}} \pi_{K+K',\,0}\right|_{\prolto{\widetilde{\ouv}'}{K+K'}}\,.
  \end{gather}
\end{proposition}
\begin{proof}
  If the above conditions on $\Phi$ and $\Psi$ are satisfied, and $x(.):I\to M$ is a solution of $\Sigma$ whose
  $K$\textsuperscript{th} jet remains inside $\ouv$, then the first part of (\ref{eq:eqgeom1}) implies that
  $\mathcal{D}_\Phi^{K}(\,x(.)\,)$ is a solution of $\Sigma'$, the first part of (\ref{eq:eqgeom2}) implies that its
  $K$\textsuperscript{th} jet remains inside $\ouv'$, and the first part of (\ref{eq:eqgeom3}) implies that
  $\mathcal{D}_\Psi^{K'}(\,\mathcal{D}_\Phi^K(\,x(.)\,)\,)\ =\ x(.)$. This proves the first item of Definition~\ref{def-eq}; the
  second item follows in the same way from the second part of (\ref{eq:eqgeom1}), (\ref{eq:eqgeom2}) and (\ref{eq:eqgeom3}).

  Conversely, if $\Phi$ and $\Psi$ satisfy the properties of Definition~\ref{def-eq}, their restrictions $\widetilde\Phi$ and
  $\widetilde\Psi$ to $\widetilde{\ouv}$ and $\widetilde{\ouv}'$ respectively satisfy the above relations
  because through each point in $\prolto{\widetilde{\ouv}}{K+1}$, $\prolto{\widetilde{\ouv'}}{K'+1}$,
  $\prolto{\widetilde{\ouv}}{K+K'}$ or $\prolto{\widetilde{\ouv}}{K+K'}$ passes a jet of order $K+1$, $K'+1$ or $K+K'$ of a
  solution of $\Sigma$ or $\Sigma'$; differentiating yields the required relations.
\end{proof}

\smallskip

\begin{proposition}[Static Equivalence]
  \label{prop-eqst}
  With $\ouv_1\subset J^{1}(M)=\mathrm{T}M$ and $\ouv'_1\subset J^{1}(M')=\mathrm{T}M$ two open subsets, systems $\Sigma$ and
  $\Sigma'$ are \emph{static equivalent over $\ouv_1$ and $\ouv'_1$} if and only if, with $\widetilde{\ouv}_1,\widetilde{\ouv}'_1$
  defined in (\ref{eq:omegtilde}), there exist a smooth diffeomorphism $\Phi_0:\widetilde{\ouv}_0\to\widetilde{\ouv}'_0$, and its
  inverse $\Psi_0$ such that $\widetilde{\Phi}_0^{[1]}(\widetilde{\ouv}_1)=\widetilde{\ouv}'_1$ (and
  $\widetilde{\Psi}_0^{[1]}(\widetilde{\ouv}'_1)=\widetilde{\ouv}_1$).
\end{proposition}
\begin{proof}
  This is a re-phrasing of point (a) of Proposition~\ref{prop-static}.
\end{proof}

\subsection{Proof of Theorem~\ref{th-main}}
\label{sec-proof}
Assume that $\Sigma$ and $\Sigma'$ are dynamic equivalent over the open sets $\ouv\subset J^K(M)$ and $\ouv'\subset J^{K'}(M')$;
let $\widetilde{\Phi}:\widetilde{\ouv}\to M'$ and $\widetilde{\Psi}:\widetilde{\ouv}' \to M$ be the smooth maps given by
Proposition~\ref{prop-eq} (recall that $\widetilde{\ouv}$ and $\widetilde{\ouv}'$ are open subsets of $\Sigma_{K}$ and
$\Sigma'_{K'}$). We define open subsets $\widetilde{\ouv}^S\subset\widetilde{\ouv}$ and $\widetilde{\ouv}^{\prime
  S}\subset\widetilde{\ouv}'$ and state four lemmas concerning these~:\vspace{-0.5em}
\begin{gather}
  \label{eq:omegS}
  \begin{array}{ccl}
    \xi\in\widetilde{\ouv}^S&\Leftrightarrow& \mbox{There is a neighborhood $V$ of $\xi$ in $\widetilde{\ouv}$ and a smooth map}
    \\
    && \widetilde{\Phi}_0:V_0\to M'\;\mbox{such that}
    \left.\widetilde{\Phi}\right|_{V} = \widetilde{\Phi}_0\circ\pi_{K,0}\,,
  \end{array}
  \\
  \label{eq:omegS'}
  \begin{array}{ccl}
    \xi'\in\widetilde{\ouv}^{\prime S}&\Leftrightarrow& \mbox{There is a neighborhood $V'$ of $\xi'$ in $\widetilde{\ouv}'$ and a smooth map}
    \\
    && \widetilde{\Psi}_0:V'_0\to M\;\mbox{such that}
    \left.\widetilde{\Psi}\right|_{V'} = \widetilde{\Psi}_0\circ\pi_{K,0}\,.
  \end{array}
\end{gather}
\begin{lemma}
  \label{lem-ana} In the analytic case, and if $\widetilde{\ouv}=\ouv\cap\Sigma$ and $\widetilde{\ouv}'=\ouv'\cap\Sigma'$ are connected, one has either $\widetilde{\ouv}^S=\widetilde{\ouv}$ or
  $\widetilde{\ouv}^S=\varnothing$, and either $\widetilde{\ouv}^{\prime S}=\widetilde{\ouv}'$ or $\widetilde{\ouv}^{\prime
    S}=\varnothing$.
\end{lemma}
\begin{lemma}
  \label{lem-biz} One has the following identities, where the two first ones hold for any subsets $S\subset\widetilde{\ouv}$,
  $S'\subset\widetilde{\ouv}'$ and any integer $\ell$, $0\leq\ell\leq K+K'$,  \vspace{-0.5em}
\begin{gather}
    \label{eq:biz}
    \pi_{K+K',\ell}\left( {\kern0pt \widetilde{\Phi}^{[K']}}^{-1} (S')\right)=\widetilde{\Psi}^{[\ell]}\left(S'_{K'+\ell}\right),\ \;
    \pi_{K+K',\ell}\left( {\kern0pt \widetilde{\Psi}^{[K]}}^{-1} (S)\right)=\widetilde{\Psi}^{[\ell]}\left(S_{K+\ell}\right),
\\
  \label{eq:surjec}
  \vspace{-0.4em}\widetilde{\Phi}^{[1]}(\widetilde{\ouv}_{K+1})=\widetilde{\ouv}'_{1}\,,\ \ \ \ 
  \widetilde{\Psi}^{[1]}(\widetilde{\ouv}'_{K'+1})=\widetilde{\ouv}_{1}\,.
  \end{gather}
\end{lemma}
 \vspace{-0.8em}
\begin{lemma}
  \label{lem-OmegaS} \textbf{\boldmath If $n<n'$,} then $\widetilde{\ouv}^S=\varnothing$. \hspace{2em}
\textbf{\boldmath If $n>n'$,} then
  $\widetilde{\ouv}^{\prime S}=\varnothing$. 

  \textbf{\boldmath If $n=n'$,} there is, for all $\xi_K\in\widetilde{\ouv}^S$, a neighborhood $\mathcal{\ouvxiProofLem}_1$ of
  $\xi_1=\pi_{K,1}(\xi_K)$ in $\ouv_1$ and an open subset $\mathcal{\ouvxiProofLem}'_1$ of $\ouv_1'$ such that systems $\Sigma$
  and $\Sigma'$ are static equivalent over $\mathcal{\ouvxiProofLem}_1$ and $\mathcal{\ouvxiProofLem}'_1$.
  There is also, for all
  $\xi'_{K'}\in\widetilde{\ouv}^{\prime S}$, a neighborhood $\mathcal{\ouvxiProofLemB}'$ of $\xi'_1=\pi_{K',1}(\xi'_{K'})$ in $\ouv'_1$ and an open subset $\mathcal{\ouvxiProofLemB}_1$ of $\ouv_1$ such that systems $\Sigma$ 
  and $\Sigma'$ are static equivalent over $\mathcal{\ouvxiProofLemB}_1$ and $\mathcal{\ouvxiProofLemB}'_1$.
Finally,
\begin{align}
  \label{eq:Sfull}
  \ \pi_{K+K',K'}\left( {\kern0pt \widetilde{\Psi}^{[K]} }^{-1} \left(\widetilde{\ouv}^S\right) \right)
  = \widetilde{\Phi}^{[K']} \left( \widetilde{\ouv}^S_{K+K'}\right)
  & \ =\  \widetilde{\ouv}^{\prime S}\,,
\\
  \label{eq:Sfull'}
  \ \pi_{K+K',K}\left( {\kern0pt \widetilde{\Phi}^{[K']} }^{-1}\left(\widetilde{\ouv}^{\prime S}\right) \right)
  = \widetilde{\Psi}^{[K]} \left(\widetilde{\ouv}^{\prime S}_{K'+K}\right)
  & \ =\  \widetilde{\ouv}^{S}\,.
\end{align}
\end{lemma}
\begin{lemma}
  \label{lem-prerul}
  For all $\xi_{K+1}\in\widetilde{\ouv}_{K+1}$ such that
  $\xi_K=\pi_{K+1,K}(\xi_{K+1})\in\widetilde{\ouv}\setminus\widetilde{\ouv}^S$, there is a straight line in
  $\mathrm{T}_{\widetilde{\Phi}(\xi_K)}M'$ that has contact of infinite order with $\Sigma'$ at
  $\widetilde{\Phi}^{[1]}(\xi_{K+1})$.
\end{lemma}

\medskip These lemmas will be proved later. Let us finish the proof of the Theorem. 

\smallskip

If $n<n'$, (\ref{eq:surjec}) implies existence, for each $\xi'\in\widetilde{\ouv}'_{1}=\ouv_1\cap\Sigma'$, of some
$\xi_{K+1}\in\widetilde{\ouv}_{K+1}$ such that $\widetilde{\Phi}^{[1]}(\xi_{K+1})=\xi'$ and finally, since $\widetilde{\ouv}^{S}$
is empty according to Lemma~\ref{lem-OmegaS}, Lemma~\ref{lem-prerul} yields a straight line in
$\mathrm{T}_{\xi'_0}M'$ that has contact of infinite order with $\Sigma'$ at $\xi'$; from Proposition~\ref{prop-landsberg}, this
implies that system $\Sigma'$ is ruled over $\ouv_1$.
If $n>n'$, one concludes in the same way.

\smallskip

Now assume {\boldmath$n=n'$}. For all $\xi'$ in
$\widetilde{\Phi}^{[1]}\left((\widetilde{\ouv}\setminus\widetilde{\ouv}^S)_{K+1}\right)$, there is, according to Lemma~\ref{lem-prerul}, a straight line in
$\mathrm{T}_{\xi'_0}M'$ that has contact of infinite order with $\Sigma'$ at $\xi'$. 
By continuity, this is also true for all $\xi'$ in the topological closure
\begin{align}
  \label{eq:R'}
  \widetilde{{R}}'\ &
  =\ \overline{ \widetilde{\Phi}^{[1]} \left( (\widetilde{\ouv}\setminus\widetilde{\ouv}^S)_{K+1}\right)}
  = \overline{ \pi_{K+K',1}\left( {\kern0pt \widetilde{\Psi}^{[K]} }^{-1}
      \left(\widetilde{\ouv}\setminus\widetilde{\ouv}^S\right) \right) }\,,
\end{align}
where the second equality come from (\ref{eq:biz}).
Let $i(\widetilde{{R}}')$ be the interior of $\widetilde{{R}}'$ for the induced topology on $\Sigma'$; since
$\widetilde{{R}}'=\overline{i(\widetilde{{R}}')}$, there is an open subset $\mathcal{R}'$ of
$\ouv_1'\subset\mathrm{T}M'$, enjoying the property that it is the interior of its topological closure, and such that
$\mathcal{R}'\cap\Sigma'=i(\widetilde{{R}}')$ and $\overline{\mathcal{R}'}\cap\Sigma'=\widetilde{{R}}'$. 
From Proposition~\ref{prop-landsberg}, $\Sigma'$ is ruled over $\mathcal{R}'$. Setting
$\mathcal{S}'=\ouv_1'\setminus\overline{\mathcal{R}'}$, one has
$\ouv_1'=\overline{\mathcal{R}'}\cup\mathcal{S}'=\mathcal{R}'\cup\overline{\mathcal{S}'}$. 
Along the same lines, $\Sigma$ is ruled over $\mathcal{R}$, open subset of $\ouv_1\subset\mathrm{T}M$ such that $\mathcal{R}\cap\Sigma$
is the relative interior of 
\begin{align}
  \label{eq:R}\widetilde{{R}}\ & 
  =\ \overline{ \widetilde{\Psi}^{[1]} \left( (\widetilde{\ouv}'\setminus\widetilde{\ouv}^{\prime S})_{K'+1}\right)}
  = \overline{ \pi_{K+K',1}\left( {\kern0pt \widetilde{\Phi}^{[K']} }^{-1}
      \left(\widetilde{\ouv}'\setminus\widetilde{\ouv}^{\prime S}\right) \right) }\,,
\end{align}
and such that $\ouv_1=\overline{\mathcal{R}}\cup\mathcal{S}=\mathcal{R}\cup\overline{\mathcal{S}}$ with
$\mathcal{S}=\ouv_1\setminus\mathcal{R}$.

We have proved (\ref{eq:partition}) and point 1; let us prove point 2. Obviously,
\begin{displaymath}
  \mathcal{S}\cap\Sigma\subset\pi_{K+K',1}\left( {\kern0pt \widetilde{\Phi}^{[K']} }^{-1}\bigl(\widetilde{\ouv}^{\prime S}\bigr)
\right)
\ \ \mbox{and}\ \ 
  \mathcal{S}'\cap\Sigma'\subset\pi_{K+K',1}\left( {\kern0pt \widetilde{\Psi}^{[K]} }^{-1} \bigl(\widetilde{\ouv}^S\bigr) \right)\ .
\end{displaymath}
Using identities (\ref{eq:Sfull}) and (\ref{eq:Sfull'}), this implies
\begin{equation}
  \label{eq:S}
    \mathcal{S}\cap\Sigma\subset\pi_{K,1}\bigl(\,\widetilde{\ouv}^{S}\,\bigr)
\ \ \mbox{and}\ \ 
  \mathcal{S}'\cap\Sigma'\subset  \pi_{K',1}\bigl(\,\widetilde{\ouv}^{\prime S}\,\bigr)\ .
\end{equation}

For all $\xi$ in $\mathcal{S}\cap\Sigma$, there is one $\xi_K\in\widetilde{\ouv}^{S}$ such that $\xi=\pi_{K,1}(\xi_K)$ and, from
Lemma~\ref{lem-OmegaS}, a neighborhood $\mathcal{\ouvxiProofLem}_1^\xi$ of
  $\xi$ in $\ouv_1$ and an open subset $\mathcal{\ouvxiProofLem}^{\prime\,\xi}_1$ of $\ouv_1'$ such that systems $\Sigma$
  and $\Sigma'$ are static equivalent over $\mathcal{\ouvxiProofLem}_1^\xi$ and $\mathcal{\ouvxiProofLem}^{\prime\,\xi}_1$.
For all $\xi'$ in $\mathcal{S}'\cap\Sigma'$, there is one $\xi'_{K'}\in\widetilde{\ouv}'^{S}$ such that $\xi'=\pi_{K',1}(\xi'_{K'})$ and, from
Lemma~\ref{lem-OmegaS}, a neighborhood $\mathcal{\ouvxiProofLemB}^{\prime\,\xi'}$ of $\xi'_1=\pi_{K',1}(\xi'_{K'})$ in $\ouv'_1$ and an open subset $\mathcal{\ouvxiProofLemB}_1^{\xi'}$ of $\ouv_1$ such that systems $\Sigma$ 
  and $\Sigma'$ are static equivalent over $\mathcal{\ouvxiProofLemB}_1^{\xi'}$ and $\mathcal{\ouvxiProofLemB}^{\prime\,\xi}_1$.

Now, $(\mathcal{\ouvxiProofLem}_1^\xi)_{\xi\in\mathcal{S}\cap\Sigma}$ is an open covering of $\mathcal{S}\cap\Sigma$
and
$(\mathcal{\ouvxiProofLemB}_1^{\prime\,\xi'})_{\xi'\in\mathcal{S}'\cap\Sigma'}$ is an open covering of $\mathcal{S}'\cap\Sigma'$.
Take for $(\widetilde{\mathcal{S}}^\alpha)_{\alpha\in A}$ the union of 
$(\mathcal{\ouvxiProofLem}_1^\xi)_{\xi\in\mathcal{S}\cap\Sigma}$ and $(\mathcal{\ouvxiProofLemB}_1^{\xi'})_{\xi'\in\mathcal{S}'\cap\Sigma'}$;
take for $(\widetilde{\mathcal{S}}^{\prime\alpha})_{\alpha\in A}$ the union of
$(\mathcal{\ouvxiProofLem}^{\prime\;\xi}_1)_{\xi\in\mathcal{S}\cap\Sigma}$ and
$(\mathcal{\ouvxiProofLemB}_1^{\prime\,\xi'})_{\xi'\in\mathcal{S}'\cap\Sigma'}$.

This proves the smooth case, and obviously implies the real analytic one from Lemma~\ref{lem-ana}.
\qquad\endproof

\medskip \noindent
Let us now prove the four lemmas used in the above proof.

\smallskip
\textit{Proof of Lemma~\ref{lem-ana}}.  
  If $\widetilde{\ouv}^S\neq\varnothing$, then there is at least an open set in $\widetilde{\ouv}$ derivatives of $\widetilde{\Phi}$
  along any vertical vector field (preserving fibers of $\Sigma_K\to M$) are identically zero; since these are real analytic they
  must be zero all over $\widetilde{\ouv}$, assumed connected, hence $\widetilde{\ouv}^S=\widetilde{\ouv}$. The proof is similar in
  $\widetilde{\ouv}'$. 
\qquad\endproof

\bigskip \textit{Proof of Lemma~\ref{lem-biz}}.  
  The first relation in (\ref{eq:biz}) is a consequence of the two identities
  \begin{equation}
    \label{eq:pr2}
    \pi_{K'+\ell,K'}\circ\widetilde{\Phi}^{[K'+\ell]} =
    \widetilde{\Phi}^{[K']}\circ\pi_{K+K'+\ell,K+K'}\ \ \mbox{and}\ \ 
    \widetilde{\Psi}^{[\ell]}\circ\widetilde{\Phi}^{[K'+\ell]} = \pi_{K+K'+\ell,\ell}\,,
  \end{equation}
  respectively (\ref{eq:piPhi}) with $(r,s)=(K'+\ell,K')$ and the $\ell$\textsuperscript{th} prolongation of
  (\ref{eq:eqgeom3}). The second relation follows from interchanging $K,\Phi,S$ with $K',\Psi,S'$.

  From equations (\ref{eq:eqgeom1}) and (\ref{eq:eqgeom2}), one has, for any positive integer $\ell$,
  \begin{equation}
    \label{eq:pr3}
    \widetilde{\Phi}^{[\ell]}(\widetilde{\ouv}_{K+\ell})\subset\widetilde{\ouv}'_{\ell}\ \ \mbox{and}\ \ 
    \widetilde{\Psi}^{[\ell]}(\widetilde{\ouv}'_{K'+\ell})\subset\widetilde{\ouv}_{\ell}
  \end{equation}
  (for instance, (\ref{eq:eqgeom1}) implies $\widetilde{\Phi}^{[\ell]}(\widetilde{\ouv}_{K+\ell})\subset\Sigma'_\ell$,
  (\ref{eq:eqgeom2}) implies  $\widetilde{\Phi}^{[\ell]}(\widetilde{\ouv}_{K+\ell})\subset\ouv'_{\ell}$, hence the first relation
  above because $\widetilde{\ouv}'_{\ell}\!=\!\ouv'_{\ell}\cap\Sigma'_{\ell}$).
  We only need to prove the reverse inclusions for $\ell=1$. Let us do it for the second one. The second relation in
  (\ref{eq:pr2}) for $\ell=1$ implies $\widetilde{\ouv}_{1}=\widetilde{\Psi}^{[1]}\left(\widetilde{\Phi}^{[K'+1]}
    (\widetilde{\ouv}_{K+K'+1})\right)$, and finally $\widetilde{\ouv}_{1}\subset\widetilde{\Psi}^{[1]}(\widetilde{\ouv}'_{K'+1})$
  from the first relation in (\ref{eq:pr2}) with $\ell=K'+1$.
\qquad\endproof

\bigskip \textit{Proof of Lemma~\ref{lem-OmegaS}}.  
Assume for instance that $\widetilde{\ouv}^S$ is non-empty; then it contains
an open subset $V$ and there is a smooth $\widetilde{\Phi}_0:V_0\to M'$ such that, in restriction to $V$,
$\widetilde{\Phi}=\widetilde{\Phi}_0\circ\pi_{K,0}$. Hence (\ref{eq:eqgeom3}) implies, on the open subset
$V'=\left(\widetilde{\Psi}^{[K]}\right)^{-1}(V)$ of $\prolto{\Sigma'}{K+K'}$,
\begin{equation}
  \label{eq:jbh3}
  \widetilde{\Phi}_0\circ\pi_{K,0}\circ\widetilde{\Psi}^{[K]}=
  \left.
    \pi_{K+K',\,0}\right|_{V'}\,.
\end{equation}
The rank of the map on the left-hand side is $n'$ while the rank of the right-hand side is no larger than $n$ (rank of
$\pi_{K,0}$), hence $\widetilde{\ouv}^S\neq\emptyset$ implies $n'\leq n$.  By interchanging the two systems, this proves the fist
sentence of the Lemma.

\smallskip

Let us now turn to the case where $n=n'$. Consider $\xi_K$ in $\widetilde{\ouv}^S$.  By definition of $\widetilde{\ouv}^S$, there is
a neighborhood $\ouvxiProofLem$ and a smooth (real analytic in the real analytic case) map $\widetilde{\Phi}_0:{\ouvxiProofLem}_0\to M'$ such that
$\widetilde{\Phi}=\widetilde{\Phi}_0\circ\pi_{K,0}$ on $\ouvxiProofLem$. 
Let $\ouvxiProofLem'$ be defined from $\ouvxiProofLem$ as
  \begin{equation}
    \label{eq:W'}
    \ouvxiProofLem'=\pi_{K+K',K'}\left({\kern0pt\widetilde{\Psi}^{[K]}}^{-1}(\ouvxiProofLem)\right) =
    \widetilde{\Phi}^{[K']}(\ouvxiProofLem_{K+K'})\,, 
  \end{equation}
where the second equality comes from (\ref{eq:biz}). 
Applying $\widetilde{\Psi}$ and  $\widetilde{\Psi}^{[1]}$ to both sides of the first equality in (\ref{eq:piPhi}) and using (\ref{eq:W'})
with $(r,s)=(K,0)$ and $(r,s)=(K,1)$ yields 
\begin{equation}
  \label{eq:pr11}
  \widetilde{\Psi}(\ouvxiProofLem')=\ouvxiProofLem_0\,,\ \ \ \ \widetilde{\Psi}^{[1]}(\ouvxiProofLem'_{K'+1})=\ouvxiProofLem_1\,.
\end{equation}
Substituting $\widetilde{\Phi}=\widetilde{\Phi}_0\circ\pi_{K,0}$ in (\ref{eq:eqgeom3}), one has 
$\;\widetilde{\Phi}_0\circ\widetilde{\Psi}\circ\pi_{K+K',K'}=\pi_{K+K',0}$
on ${\kern0pt\widetilde{\Psi}^{[K]}}^{-1}(\ouvxiProofLem)\;$, and finally
\begin{equation}
  \label{eq:pr12}
  \widetilde{\Phi}_0\circ\widetilde{\Psi}=\pi_{K',0}\ \mbox{on}\  \ouvxiProofLem'\,;
\end{equation}
in a similar way, substituting $\widetilde{\Phi}^{[1]}=\widetilde{\Phi}_0^{[1]}\circ\pi_{K+1,1}$ in the first prolongation
of (\ref{eq:eqgeom3}),
\begin{equation}
  \label{eq:pr13}
  \widetilde{\Phi}_0^{[1]}\circ\widetilde{\Psi}^{[1]}=\pi_{K'+1,1}\ \mbox{on}\  \ouvxiProofLem'_{K'+1}\,.
\end{equation}
Applying $\widetilde{\Phi}_0$ to both sides of the first relation and $\widetilde{\Phi}_0^{[1]}$ to both sides of the second
relation in (\ref{eq:pr11}), one has, using (\ref{eq:pr12}) and (\ref{eq:pr13}),
\begin{equation}
  \label{eq:pr14}
  \widetilde{\Phi}_0(\ouvxiProofLem_0)=\ouvxiProofLem'_0\,,\ \ \ \widetilde{\Phi}_0^{[1]}(\ouvxiProofLem_1)=\ouvxiProofLem'_1\,.
\end{equation}
Since the rank of $\pi_{K',0}$ in the right-hand side of (\ref{eq:pr12}) is $n'=n$ at all points of $\ouvxiProofLem'$, 
$\widetilde{\Phi}_0$ must be a local diffeomorphism at all point of $\widetilde{\Psi}(\ouvxiProofLem')=\ouvxiProofLem_0$ and in
particular at $\xi_0$: 
by the inverse function theorem, there is a neighborhood $O$ of $\xi_0=\pi_{K,0}(\xi)$ in $\ouvxiProofLem_0$ and a neighborhood
$O'$ of $\Phi_0(\xi_0)$ in $M'$ such that $\Phi_0$ defines a diffeomorphism $O\to O'$.  

Let us now replace $\ouvxiProofLem$ with $\ouvxiProofLem\cap{\pi_{K,0}}^{-1}(O)$, a smaller neighborhood of $\xi_K$; 
$\ouvxiProofLem'$ is still defined by (\ref{eq:W'}) from this smaller $\ouvxiProofLem$, one has $\ouvxiProofLem_0=O$,
the former $\widetilde{\Phi}_0$ is replaced by its restriction to this smaller $\ouvxiProofLem_0$, and the above relations still
hold.
In particular, $O'=\widetilde{\Phi}_0(O)$ must be all $\ouvxiProofLem'_0$ according to (\ref{eq:pr14}), i.e.
$\widetilde{\Phi}_0$ defines a diffeomorphism
$\ouvxiProofLem_0\to\ouvxiProofLem'_0$; let $\widetilde{\Psi}_0$ be its inverse. 
Composing each side of (\ref{eq:pr12}) with $\widetilde{\Psi}_0$, one gets $\widetilde{\Psi}=\widetilde{\Psi}_0\circ\pi_{K',0}$ on
$\ouvxiProofLem'$; hence, by (\ref{eq:omegS'}), one has $\ouvxiProofLem'\subset\widetilde{\ouv}^{\prime S}$ and, since this is
true for all $\xi_K$ in $\widetilde{\ouv}^S$, one has 
\begin{equation}
  \label{eq:pr15}
  \pi_{K+K',K'}\left({\kern0pt\widetilde{\Psi}^{[K]}}^{-1}(\widetilde{\ouv}^{S})\right) =
    \widetilde{\Phi}^{[K']}(\widetilde{\ouv}^{S}_{K+K'})\subset\widetilde{\ouv}^{\prime S}\,.
\end{equation}
Let $\mathcal{\ouvxiProofLem}_1$ and
$\mathcal{\ouvxiProofLem}'_1$ and be open subsets of $\ouv_1$ and $\ouv'_1$ such that
\begin{equation}
  \label{eq:out}
  \ouvxiProofLem_1=\Sigma\cap\mathcal{\ouvxiProofLem}_1\,,\ \ \ 
\ouvxiProofLem'_1=\Sigma\cap\mathcal{\ouvxiProofLem}'_1\ .
\end{equation}
From Proposition~\ref{prop-eqst}, the second relation in (\ref{eq:pr14}) implies that systems $\Sigma$ and $\Sigma'$ are static
equivalent over $\mathcal{\ouvxiProofLem}_1$ and $\mathcal{\ouvxiProofLem}'_1$.
Interchanging the two systems, one proves that
\begin{equation}
  \label{eq:pr16}
  \pi_{K+K',K}\left({\kern0pt\widetilde{\Phi}^{[K']}}^{-1}(\widetilde{\ouv}^{\prime S})\right) =
    \widetilde{\Psi}^{[K]}(\widetilde{\ouv}^{\prime S}_{K+K'})\subset\widetilde{\ouv}^{S}\,.
\end{equation}
and that, for all $\xi'_{K'}\in\widetilde{\ouv}^{\prime S}$, there are a neighborhood $\mathcal{\ouvxiProofLemB}'$ of
$\xi'_1=\pi_{K',1}(\xi'_{K'})$ in $\ouv'_1$ and an open subset $\mathcal{\ouvxiProofLemB}_1$ of $\ouv_1$ such that systems
$\Sigma$ and $\Sigma'$ are static equivalent over $\mathcal{\ouvxiProofLemB}_1$ and $\mathcal{\ouvxiProofLemB}'_1$.

Now, $\widetilde{\Phi}^{[K']}(\widetilde{\ouv}^{S}_{K+K'})\subset\widetilde{\ouv}^{\prime S}$ in (\ref{eq:pr15}) implies
$\widetilde{\ouv}^{S}_{K+K'}\subset{\kern0pt\widetilde{\Phi}^{[K']}}^{-1}(\widetilde{\ouv}^{\prime S})$, and hence 
$\widetilde{\ouv}^{S}\subset\pi_{K+K',K}\left({\kern0pt\widetilde{\Phi}^{[K']}}^{-1}(\widetilde{\ouv}^{\prime S})\right)$. Hence
(\ref{eq:pr15}) implies the converse inclusion in (\ref{eq:pr16}); in a similar way (\ref{eq:pr16}) implies the converse inclusion
in (\ref{eq:pr15}).  This proves (\ref{eq:Sfull}) and (\ref{eq:Sfull'}), and ends the proof of Lemma~\ref{lem-OmegaS}.
\qquad\endproof

\bigskip \textit{Proof of Lemma~\ref{lem-prerul}}.  Denote by $\bar{\xi}_{K+1}$ the point $\xi_{K+1}$
in the lemma statement and set $\bar{\xi}_K=\pi_{K+1,K}({\bar\xi}_{K+1})\in\widetilde{\ouv}\setminus\widetilde{\ouv}^S$,
$\bar\xi_0=\pi_{K,0}(\bar\xi_{K+1})$, $\bar\xi_1=\pi_{K,1}(\bar\xi_{K+1})$. 
From Proposition~\ref{prop-prol}, and after possibly shrinking $\ouvxiBig_K$ so that it is contained in $\ouv$, there exist a
neighborhood $\ouvxiBig_K\subset\ouv$ of $\bar{\xi}_K$ in $J^K(M)$, coordinates $(x_{\mathrm{I}},x_{\mathrm{I\!I}})$ on
$\ouvxiBig_0=\pi_{K,0}(\ouvxiBig_K)$ inducing coordinates
$(x_{\mathrm{I}},x_{\mathrm{I\!I}},\dot{x}_{\mathrm{I}},\dot{x}_{\mathrm{I\!I}},\ldots,x_{\mathrm{I}}^{(K)},x_{\mathrm{I\!I}}^{(K)})$
on $\ouvxiBig_K$, and an open subset $\ouvgraph_{K}\subset\RR^{n+Km}$ such that the equations of
$\widetilde{\ouvxiBig}_K=\ouvxiBig_K\cap\prolto{\Sigma}{K}$ in $J^K(M)$ in these coordinates are
\begin{equation}
  \label{eq:eqprolong1}
  \begin{array}{l}
    x_{\mathrm{I}}^{(i)}=f^{(i-1)}(x_{\mathrm{I}},x_{\mathrm{I\!I}},\dot{x}_{\mathrm{I\!I}},\ldots,x_{\mathrm{I\!I}}^{(i)})\,,\ \ \
    1\leq i\leq K\,,
    \\
    (x_{\mathrm{I}},x_{\mathrm{I\!I}},\dot{x}_{\mathrm{I\!I}},\ldots,x_{\mathrm{I\!I}}^{(K)})\;\in\;\ouvgraph_K\,.
  \end{array}
\end{equation}
By substitution, there is a unique smooth map $\phi_K:\ouvgraph_K\to M'$ such that
$\widetilde{\Phi}(\xi)=\phi_K(x_{\mathrm{I}},x_{\mathrm{I\!I}},\dot{x}_{\mathrm{I\!I}},\ldots,x_{\mathrm{I\!I}}^{(K)})$ for all
$\xi$ in $\widetilde{\ouvxiBig}_K$ with coordinate vector
$(x_{\mathrm{I}},x_{\mathrm{I\!I}},\ldots,x_{\mathrm{I}}^{(K)},x_{\mathrm{I\!I}}^{(K)})$. 

Let $\overline{X}_i= (\overline{x_{\mathrm{I}}},\overline{x_{\mathrm{I\!I}}},\overline{\dot{x}_{\mathrm{I}}},
\overline{\dot{x}_{\mathrm{I\!I}}}, \ldots,\overline{x_{\mathrm{I}}^{(i)}},\overline{x_{\mathrm{I\!I}}^{(i)}})$ be the
coordinate vector of $\bar{\xi}_i$ for $i\leq K+1$ and $\bar\rho$ the smallest integer such that $\phi_K$ does not depend on 
$x_{\mathrm{I\!I}}^{(\bar\rho+1)},\ldots,x_{\mathrm{I\!I}}^{(K)}$ on at least one neighborhood of $\overline{X}_K$.
Shrinking $\ouvgraph_K$ to this neighborhood, and
$\widetilde{\ouvxiBig}_K$ accordingly, we may define $\phi:\ouvgraph_{\bar\rho}\to M'$, with $\ouvgraph_{\bar\rho}$ the projection
of $\ouvgraph_K$ on $\RR^{n+\bar\rho\,m}$, such that 
$\widetilde{\Phi}(\xi)=\phi_K(x_{\mathrm{I}},x_{\mathrm{I\!I}},\dot{x}_{\mathrm{I\!I}},\ldots,x_{\mathrm{I\!I}}^{(K)})
=\phi(x_{\mathrm{I}},x_{\mathrm{I\!I}},\dot{x}_{\mathrm{I\!I}},\ldots,x_{\mathrm{I\!I}}^{(\bar\rho)})$.
If $\bar\rho$ was zero, one would have $\widetilde{\Phi}(\xi)=\phi(x_{\mathrm{I}},x_{\mathrm{I\!I}})$, hence the right-hand side
of (\ref{eq:omegS}) would be satisfied for $\xi=\bar\xi_K$ with $V=\widetilde{\ouvxiBig}_K$; this is impossible because we assumed
$\bar{\xi}_K\in\widetilde{\ouv}\setminus\widetilde{\ouv}^S$. Hence $\bar\rho\geq1$.

For all $\xi_{K+1}$ in $\widetilde{\ouvxiBig}_{K+1}$ with coordinate vector
$(x_{\mathrm{I}},x_{\mathrm{I\!I}},\ldots,x_{\mathrm{I\!I}}^{(K+1)})$, one has
\begin{equation}
\label{eq:bof3}
  \widetilde{\Phi}^{[1]}(\xi_{K+1})=\chi(x_{\mathrm{I}},x_{\mathrm{I\!I}},\dot{x}_{\mathrm{I\!I}},\ldots,x_{\mathrm{I\!I}}^{(\bar\rho)},x_{\mathrm{I\!I}}^{(\bar\rho+1)})
\end{equation}
with $\chi:\ouvgraph_{\bar\rho+1}\to\mathrm{T}M'$ the map defined by
\begin{align}
  \label{eq:chi}
  \chi(x_{\mathrm{I}}\ldots x_{\mathrm{I\!I}}^{(\bar\rho+1)})=
  \Bigl(\;\phi(x_{\mathrm{I}}  \ldots x_{\mathrm{I\!I}}^{(\bar\rho)})\;,\;
  a(x_{\mathrm{I}}\ldots x_{\mathrm{I\!I}}^{(\bar\rho)}) + \,\frac{\partial\phi}{\partial
    x_{\mathrm{I\!I}}^{(\bar\rho)}}(x_{\mathrm{I}}\ldots x_{\mathrm{I\!I}}^{(\bar\rho)})\;x_{\mathrm{I\!I}}^{(\bar\rho+1)}\;
  \Bigr)
\end{align}
with $a= \frac{\partial\phi}{\partial x_{\mathrm{I}}}f
  +\sum_{i=0}^{\bar\rho-1} \frac{\partial\phi}{\partial
    x_{\mathrm{I\!I}}^{(i)}}\,x_{\mathrm{I\!I}}^{(i+1)}$.
According to (\ref{eq:eqgeom2}), (\ref{eq:eqprolong1}) and (\ref{eq:bof3}), $\Sigma'$ contains
$\chi(\ouvgraph_{\bar\rho+1})$.
Now, for any $(x_{\mathrm{I}},\ldots,x_{\mathrm{I\!I}}^{(\bar\rho+1)})\in\ouvgraph_{\bar\rho+1}$ such that the linear map
\begin{equation*}
  \frac{\partial\phi}{\partial x^{(\rho)}_{\mathrm{I\!I}}}(x_{\mathrm{I}},\ldots, x_{\mathrm{I\!I}}^{(\bar\rho)}):\
  \ \RR^m\;\to\; \mathrm{T}_{\phi(x_{\mathrm{I}},\ldots, x_{\mathrm{I\!I}}^{(\bar\rho)})}M'
\end{equation*}
is nonzero, picking $\underline{w}\neq0$ in its range, (\ref{eq:chi}) implies that the straight line $\Delta$
in $\mathrm{T}_{\phi(x_{\mathrm{I}},\ldots, x_{\mathrm{I\!I}}^{(\bar\rho)})}M'$ passing through $\chi(x_{\mathrm{I}}\ldots
x_{\mathrm{I\!I}}^{(\bar\rho+1)})$ with direction $\underline{w}$ has a segment around $\chi(x_{\mathrm{I}}\ldots
x_{\mathrm{I\!I}}^{(\bar\rho+1)})$ contained in $\Sigma'$, hence in particular $\Delta$ has contact of infinite order with
$\Sigma'$ at point $\chi(x_{\mathrm{I}},\ldots, x_{\mathrm{I\!I}}^{(\bar\rho+1)})$.  
To sum up, we have proved so far that, for all $\xi_{K+1}$ in $\widetilde{\ouvxiBig}_{K+1}$ with coordinate vector
$(x_{\mathrm{I}},x_{\mathrm{I\!I}},\ldots,x_{\mathrm{I\!I}}^{(K+1)})$ such that $\frac{\partial\phi}{\partial
  x^{(\rho)}_{\mathrm{I\!I}}}(x_{\mathrm{I}}\ldots x_{\mathrm{I\!I}}^{(\bar\rho)})$ is nonzero, there is a straight line $\Delta_{\xi_{K+1}}$
in $\mathrm{T}_{\widetilde{\Phi}(\xi_K)}M'$ passing through $\widetilde{\Phi}^{[1]}(\xi_{K+1})$ 
that has contact of infinite order with $\Sigma'$ at $\widetilde{\Phi}^{[1]}(\xi_{K+1})$. The set of such points $\xi_{K+1}$ may
not contain $\bar{\xi}_{K+1}$ but its topological closure does, by minimality of $\bar\rho$; taking a sequence of points
$\xi_{K+1}$ that converges to $\bar{\xi}_{K+1}$, any accumulation point of the compact sequence $\left(\Delta_{\xi_{K+1}}\right)$
is a straight line in $\mathrm{T}_{\widetilde{\Phi}(\bar{\xi}_K)}M'$ passing through $\widetilde{\Phi}^{[1]}(\bar{\xi}_{K+1})$ 
that has contact of infinite order with $\Sigma'$ at $\widetilde{\Phi}^{[1]}(\bar{\xi}_{K+1})$.
\qquad\endproof

\subsection*{Acknowledgements} The authors is indebted to the anonymous reviewer who suggested to include a sketch of proof
(Section~\ref{sec-sketch}) that clarifies the main ideas.


\begin{thebibliography}{10}

\bibitem{Ande-Ibr79}
{\sc R.~L. Anderson and N.~H. Ibragimov}, {\em Lie-B{\"a}cklund transformations
  in applications}, vol.~1 of SIAM Studies in Applied Mathematics, SIAM, 1979.

\bibitem{Avan05th}
{\sc D.~Avanessoff}, {\em Linéarisation dynamique des systèmes non linéaires et
  paramétrage de l'ensemble des solutions}, PhD thesis, Univ. de Nice - Sophia
  Antipolis, June 2005.

\bibitem{Avan-Pom07}
{\sc D.~Avanessoff and J.-B. Pomet}, {\em Flatness and {M}onge parameterization
  of two-input systems, control-affine with 4 states or general with 3 states},
  ESAIM Control Optim. Calc. Var., 13 (2007), pp.~237--264.

\bibitem{Cart14abs}
{\sc {\'E}.~Cartan}, {\em Sur l'{\'e}quivalence absolue de certains
  syst{\`e}mes d'{\'e}quations diff{\'e}rentielles et sur certaines familles de
  courbes}, Bull. Soc. Math. France, 42 (1914), pp.~12--48.

\bibitem{Char-Lev-Mar91}
{\sc B.~Charlet, J.~L{\'e}vine, and R.~Marino}, {\em Sufficient conditions for
  dynamic state feedback linearization}, SIAM J. Control Optim., 29 (1991),
  pp.~38--57.

\bibitem{Flie-Lev-Mar-R92cras}
{\sc M.~Fliess, J.~L{\'e}vine, P.~Martin, and P.~Rouchon}, {\em Sur les
  syst{\`e}mes non lin{\'e}aires diff{\'e}rentiellement plats}, C.~R. Acad.
  Sci. Paris, S{\'e}rie I, 315 (1992), pp.~619--624.

\bibitem{Flie-Lev-Mar-R99geo}
\leavevmode\vrule height 2pt depth -1.6pt width 23pt, {\em A {L}ie-{B}\"acklund
  approach to equivalence and flatness of nonlinear systems}, IEEE Trans.
  Automat. Control, 44 (1999), pp.~922--937.

\bibitem{Golu-Gui73}
{\sc M.~Golubitsky and V.~Guillemin}, {\em Stable mappings and their
  singularities}, Springer-Verlag, New York, 1973.
\newblock Graduate Texts in Mathematics, Vol. 14.

\bibitem{Gour05}
{\sc E.~Goursat}, {\em Sur le probl\`eme de {M}onge}, Bull. Soc. Math. France,
  33 (1905), pp.~201--210.

\bibitem{Hilb12}
{\sc D.~Hilbert}, {\em {\"U}ber den begriff der klasse von
  differentialgleichungen}, Math. Ann., 73 (1912), pp.~95--108.

\bibitem{Isid-Moo-DeL86}
{\sc A.~Isidori, C.~H. Moog, and A.~de~Luca}, {\em {A} sufficient condition for
  full linearization via dynamic state feedback}, in Proc. 25th IEEE Conf. on
  Decision \& Control, Athens, 1986, pp.~203--207.

\bibitem{Jaku94}
{\sc B.~Jakubczyk}, {\em Equivalence of differential equations and differential
  algebras}, Tatra Mt. Math. Publ., 4 (1994), pp.~125--130.
\newblock Equadiff 8 (Bratislava, 1993).

\bibitem{Kupk91}
{\sc I.~Kupka}, {\em On feedback equivalence}, in Differential geometry, global
  analysis, and topology (Halifax, NS, 1990), vol.~12 of CMS Conf. Proc., Amer.
  Math. Soc., Providence, RI, 1991, pp.~105--117.

\bibitem{Land99}
{\sc J.~M. Landsberg}, {\em Is a linear space contained in a submanifold? {O}n
  the number of derivatives needed to tell}, J. Reine Angew. Math., 508 (1999),
  pp.~53--60.

\bibitem{Mart92th}
{\sc P.~Martin}, {\em Contribution {\`a} l'{\'e}tude des syst{\`e}mes
  differentiellement plats}, PhD thesis, L'{\'E}cole Nationale Sup{\'e}rieure
  des Mines de Paris, 1992.

\bibitem{Pome95vars}
{\sc J.-B. Pomet}, {\em A differential geometric setting for dynamic
  equivalence and dynamic linearization}, Banach Center Publications, 32
  (1995), pp.~319--339.
\newblock Title of the volume: "Geometry in Nonlinear Control and Differential
  Inclusions".

\bibitem{Rouc94}
{\sc P.~Rouchon}, {\em Necessary condition and genericity of dynamic feedback
  linearization}, J. of Math. Systems, Estimation, and Control, 4 (1994),
  pp.~1--14.

\bibitem{Shad90}
{\sc W.~Shadwick}, {\em Absolute equivalence and dynamic feedback
  linearization}, Syst. \& Control Lett., 15 (1990), pp.~35--39.

\bibitem{Shad-Slu94}
{\sc W.~F. Shadwick and W.~M. Sluis}, {\em Dynamic feedback for classical
  geometries}, in Differential geometry and mathematical physics (Vancouver,
  BC, 1993), vol.~170 of Contemp. Math., Amer. Math. Soc., Providence (USA),
  1994, pp.~207--213.

\bibitem{Slui93}
{\sc W.~M. Sluis}, {\em A necessary condition for dynamic feedback
  linearization}, Syst. \& Control Lett., 21 (1993), pp.~277--283.

\bibitem{Tcho87}
{\sc K.~Tcho\'{n}}, {\em The only stable normal forms of affine systems under
  feedback are linear}, Syst. \& Control Lett., 8 (1987), pp.~359--365.

\bibitem{Wilk99}
{\sc G.~R. Wilkens}, {\em Centro-affine geometry in the plane and feedback
  invariants of two-state scalar control systems}, in Differential geometry and
  control (Boulder, CO, 1997), vol.~64 of Proc. Sympos. Pure Math., Amer. Math.
  Soc., Providence, RI, 1999, pp.~319--333.

\end{thebibliography}
\end{document}